\documentclass[10pt]{amsart}
\usepackage{amsmath,amssymb,amsfonts}
\usepackage{enumerate}
\usepackage[pagebackref=true]{hyperref}
\hypersetup{
   colorlinks=true,
    linkcolor=blue,
   filecolor=magenta,      
    urlcolor=cyan,
    pdftitle={Overleaf Example},
    pdfpagemode=FullScreen,
    }
\usepackage{color}
\usepackage[dvipsnames]{xcolor}
\definecolor{sz}{RGB}{115,45,2}

\newcommand{\tb}[1]{\textcolor{sz}{\bf #1}}
\usepackage{geometry}
\usepackage{tikz}
\usepackage{tikz-cd}
\usepackage{mathtools}

\usepackage{lipsum}

\usepackage{epsfig}  		
\usepackage{epic,eepic}  
\usepackage{graphicx}
\usepackage{verbatim}
\usepackage{mathdots}
\usepackage{leftidx}

\newcommand{\bZ}{\mathbb{Z}}

\theoremstyle{plain}
\newtheorem{thm}{Theorem}[section]
\newtheorem{cor}[thm]{Corollary}
\newtheorem{lem}[thm]{Lemma}
\newtheorem{prop}[thm]{Proposition}

\newtheorem{assumption}[thm]{Assumption}

\newtheorem{defn}[thm]{Definition}

\theoremstyle{remark}

\newtheorem{remark}[thm]{Remark}

\DeclareMathOperator{\Hom}{Hom}

\begin{document}
\title{Stability of the exterior cube $\gamma$-factors for $\mathrm{GL}(6)$}

\author{Taiwang  Deng}
\address[T.D.]{Beijing Institute of Mathematical Sciences and Applications (BIMSA), Huairou District, 100084, Beijing\\
China}
\email{dengtaiw@bimsa.cn}

\author{Dongming She}
\address[D.S.]{Beijing Institute of Mathematical Sciences and Applications (BIMSA), Huairou District, 100084, Beijing\\
China}
\email{shedongming@bimsa.cn}



\begin{abstract}
We prove the stability of the Langlands-Shahidi local $\gamma$-factor for the exterior cube representation of $\mathrm{GL}_6$. More precisely, if $\pi_1$ and $\pi_2$ are irreducible admissible generic representations of $\mathrm{GL}_6(F)$ with the same central character, then
\[
\gamma(s,\pi_1\otimes\chi,\wedge^3,\psi)=
\gamma(s,\pi_2\otimes\chi,\wedge^3,\psi)
\]
for every sufficiently ramified character $\chi$ of $F^\times$, where $\chi$ is regarded as a character of $\mathrm{GL}_6(F)$ through the determinant. The proof uses the realization of the exterior cube representation by the maximal parabolic subgroup of the simply connected group of type $E_6$. We give an explicit description of the relevant geometric quotient $U_M\backslash N'$, compute its invariant measure, and relate Shahidi's partial Bessel functions to partial Bessel integrals on the Levi subgroup. The desired stability then follows from an asymptotic expansion of these partial Bessel integrals and the vanishing of highly ramified Mellin transforms.
\end{abstract}

\maketitle
\tableofcontents
\section{Introduction}
Let $F$ be a $p$-adic field and let $\psi$ be a non-trivial additive character of $F$. Many local factors attached to representations of reductive groups over $F$ satisfy a stability property: after twisting by sufficiently ramified characters, the factor depends only on the central character of the representation. This phenomenon goes back to Deligne's work on local constants \cite{Deligne73} and to the lemma of Jacquet and Shalika on highly ramified $\epsilon$-factors \cite{Jac85}.

In the Langlands-Shahidi method, stability is studied through local coefficients and their expression as Mellin transforms of partial Bessel functions. This approach was developed by Shahidi \cite{Sha02,Sha10} and has been used in several cases, including the stability of $\gamma$-factors for quasi-split groups \cite{Cog08} and the exterior and symmetric square factors for general linear groups \cite{CST17}. In the setting of the generalized doubling method, Cai, Friedberg, and Kaplan constructed Rankin-Selberg local $\gamma$-factors attached to classical groups and general linear groups, and proved that, for generic representations, these factors agree with those constructed by the Langlands-Shahidi method \cite[Corollary 4.5]{Cai2022TheGD}. They also proved stability in these cases \cite[Lemma 7.3]{Cai2022TheGD}. For the same Rankin-Selberg factors, stability also follows from \cite[Appendix A]{AGIKM}. We refer to \cite[Introduction]{DS25} for further references.

The purpose of this paper is to prove the corresponding stability result for the exterior cube $\gamma$-factor of $\mathrm{GL}_6$. This case is not covered by the preceding results and is expected to be useful in questions related to the exterior cube functorial lift
\[
\mathrm{GL}_6 \longrightarrow \mathrm{GL}(\wedge^3\mathbb{C}^6)\simeq \mathrm{GL}_{20}.
\]
We state the main result explicitly.

\begin{thm}\label{main-theorem}
Let $\pi_1$ and $\pi_2$ be irreducible admissible $\psi$-generic representations of $\mathrm{GL}_6(F)$ with the same central character. Then, for every sufficiently ramified character $\chi$ of $F^\times$,
\[
\gamma(s,\pi_1\otimes\chi,\wedge^3,\psi)=
\gamma(s,\pi_2\otimes\chi,\wedge^3,\psi),
\]
where $\chi$ is viewed as a character of $\mathrm{GL}_6(F)$ through the determinant.
\end{thm}

By the standard reduction using the structure theory of generic representations and the multiplicativity of local factors, it is enough to prove Theorem \ref{main-theorem} for supercuspidal representations. This is the case treated in the body of the paper.

The exterior cube representation of $\mathrm{GL}_6(\mathbb{C})$ appears naturally in the Langlands-Shahidi method from the simply connected split group $G$ of type $E_6$. Let $P=MN$ be the standard maximal parabolic subgroup of $G$ obtained by removing the simple root $\alpha_2$. The derived group of $M$ is isomorphic to $\mathrm{SL}_6$, and the adjoint action of the dual Levi group on the Lie algebra of the dual unipotent radical decomposes as
\[
r_1\oplus r_2,\qquad r_1\simeq \wedge^3\mathbb{C}^6,
\]
where $r_2$ is one-dimensional. If $\pi$ is a supercuspidal generic representation of $\mathrm{GL}_6(F)$ and $\sigma$ is the corresponding representation of $M(F)$ obtained from the map
\[
\xi:M\longrightarrow \mathrm{GL}_1\times \mathrm{GL}_6,
\]
then the local coefficient satisfies
\[
C_\psi(s,\sigma)=
\gamma(s,\pi^\vee,\wedge^3,\psi^{-1})
\gamma(2s,\omega_{\pi^\vee},\psi^{-1}).
\]
Thus, for two representations with the same central character, the one-dimensional Tate factor is the same after twisting. Up to replacing the representations, the twisting character, and the additive character by their inverses, the stability of the exterior cube factor follows from the stability of the corresponding local coefficients.

The above realization also has a useful interpretation in terms of $L$-monoids. In Section \ref{GMpair}, we show that the map $\xi$ lands in the $L$-monoid attached to the exterior cube representation, and that the generalized determinant on this monoid pulls back to the character $\gamma_0$ of $M$. This compatibility explains the role of $\gamma_0$ in the normalization of the local coefficient and is one point where the exterior cube case has a structure not visible in the earlier stability arguments.

The next issue is to make Shahidi's local coefficient formula explicit. The formula is an integral over the $F$-points of a geometric quotient $U_M\backslash N'$, where $U_M=U\cap M$ acts on the unipotent radical $N$ by conjugation and $N'\subset N$ is a suitable Zariski open dense subset. Unlike in several preceding cases, the orbit structure is not most efficiently determined by direct root-by-root calculations. Instead, in Theorem \ref{orbit-space-representative}, we use invariant theory, more precisely Brion's description of the invariants for the exterior cube representation, to construct an explicit slice for the relevant open quotient. This gives representatives
\[
\prod_{i=0}^5 x_{\gamma_i}(c_i),\qquad (c_0,\ldots,c_5)\in F^6,
\]
for a specific set of six roots $\gamma_i$ in the unipotent radical of the $E_6$ parabolic. After shrinking to a dense open subset, the action has trivial stabilizers and the quotient is realized as an open subset of affine $6$-space in these coordinates.

For the stability argument, the invariant measure on this quotient has to be known explicitly. We compute the Jacobian of the orbit map
\[
\Phi:U_M\times F^6\longrightarrow N,\qquad (u,\dot n)\mapsto u\dot n u^{-1},
\]
and obtain the quotient measure
\[
|c_3|^2|c_4|^4|c_5|^9\prod_{i=0}^5dc_i
\]
on $U_M\backslash N'$. After passing further to $Z_M^0U_M\backslash N'$, we set $c_5=1$ and work on an open subset $R'\subset (F^\times)^5$ with measure
\[
d\dot n=|c_3|^2|c_4|^4\prod_{i=0}^4dc_i.
\]
This explicit measure is one of the ingredients which allows the local coefficient formula to be compared with Mellin transforms over tori.

The second geometric input is the Bruhat decomposition
\[
\dot w_0^{-1}n=mn'\overline n,
\]
which defines a rational map $\mathcal{N}:N'\rightarrow M$. For the above representatives of $U_M\backslash N'$, the resulting formulas are considerably more complicated than in the cases previously treated by similar methods. In Proposition \ref{BruhatDecomp}, the explicit Bruhat decomposition and the formula for $\mathcal{N}$ were first found by a computation in Magma \cite{Magma97} and are then verified by an exact Chevalley-group calculation. The induced map on $U_M\backslash N'$ is \'etale of degree two onto its image, and after passing to the quotient by $Z_M^0$ it becomes an isomorphism onto its image over a further dense open subset. This makes it possible to replace the partial Bessel functions appearing in Shahidi's quotient integral by partial Bessel integrals on $M$.

To define the partial Bessel functions in the required form, one also needs suitable open compact cutoff subgroups in $N(F)$ and $\overline{N}(F)$. In Section \ref{partial-Bessel-function}, we construct these subgroups explicitly from Chevalley root subgroups, using Moy--Prasad depth notation to record the root depths. This provides a uniform way to control the support conditions entering Shahidi's partial Bessel functions and to compare them with partial Bessel integrals on $M$.

Finally, we apply an asymptotic expansion of partial Bessel integrals along Bruhat cells of the Levi subgroup. The leading term in this expansion depends only on the central character, and therefore cancels when comparing two representations with the same central character. The remaining terms are uniformly smooth in the torus variables. After twisting by a sufficiently ramified character, the corresponding inner Mellin integrals vanish. This proves the equality of the local coefficients for the two twists, and hence the stability of the exterior cube $\gamma$-factors.

The paper is organized as follows. Section 2 recalls notation, intertwining operators, and local coefficients. Section 3 begins with the structure of the relevant $E_6$ parabolic, the realization of the exterior cube representation, and the relation with the corresponding $L$-monoid. We then describe the quotient $U_M\backslash N'$ by invariant-theoretic methods, compute its invariant measure, and analyze the Bruhat map to $M$. The remaining subsections construct the partial Bessel functions using Chevalley root-subgroup filtrations modeled on Moy--Prasad depths, establish the local coefficient formula, prove the asymptotic expansion of partial Bessel integrals, and finally prove Theorem \ref{main-theorem}.

\subsection*{Acknowledgements}
\addtocontents{toc}{\protect\setcounter{tocdepth}{1}}We express our deep gratitude to Prof. Freydoon Shahidi, whose pioneering works  \cite{Sha90}, \cite{Sha02} profoundly influenced this research. 

We are grateful to Prof. Sandeep Verma for helpful discussions on the Heisenberg structure of the unipotent radical $N$ of the parabolic subgroup $P$.

T. Deng is supported by Beijing Natural Science Foundation, No. 1244042 and National Natural Science Foundation of
China, No. 12401013.

\section{Notations and preliminaries}
\subsection{Notations}
\begin{itemize}
\item Let $F$ be a p-adic field with ring of integers $\mathcal{O}_F$. We use $\varpi$ to denote a uniformizer of $F$, and $q_F$ the residue cardinality.
    \item We use letters $G, M, N, U$, etc. to denote both a split algebraic group over a p-adic field $F$ and its group of $F$-points. Unless there is ambiguity, we put $G(F), M(F), U(F)$ for the group of $F$-points to distinguish them. 
    \item We fix an $F$-splitting
    $(B,A, \{x_\alpha: \alpha\in \Delta\})$ of $G$, where $B=AU$ is a 
    Borel subgroup of $G$, $A$ is the maximal (split) torus, $U$ is the unipotent radical of $B$, $\Delta$ is the set of simple roots determined by $B$, and for each root $\alpha$, $x_\alpha: \mathbb{G}_a\rightarrow U_{\alpha}$ is the 1-parameter subgroup associated to $\alpha$.
    \item Denote by $\Phi$(resp. $\Phi^+$, $\Phi^-$) the set of roots(resp. positive roots, negative roots) in $G$. For a standard parabolic subgroup $P=MN$ of $G$, we use $\Phi_M^+$ and $\Phi_M^-$(resp. $\Phi_N$ and $\Phi_N^-$) to denote the set of positive and negative roots in $M$ (resp. roots in $N$ and  $\overline{N}$, where $\overline{N}$ is the opposite of $N$, i.e., the unipotent subgroup generated by negative roots corresponding to the ones in $N$.)
    \item Denote by $W(G,A)$ the Weyl group of $G$. We use $w_G$(resp. $w_M$) to denote the long Weyl group element for $G$(resp. $M$). For $w\in W(G,A)$, $\dot{w}$ denotes a representative of $w$ in $G$. We use $e$ to denote the identity element in $W(G,A)$.
    \item Throughout the paper, we fix a non-trivial additive character $\psi$ of $F$.
\end{itemize}

\subsection{Intertwining operators and local coefficients} We will introduce the theory of intertwining operators and local coefficients in this section. Since all groups in this paper are split over $F$, we will only focus on the split case. 
\subsubsection{Intertwining operators}Let $G$ be a split connected reductive group over $F$ with a fixed $F$-splitting $(B,A, \{x_\alpha: \alpha\in \Delta\})$. Through the splitting, $\psi$ defines a generic character of $U$, we will still denote it by $\psi$. Then the standard parabolic subgroups are in bijection with subsets of $\Delta$. Suppose $\theta,\theta'\subset \Delta$ and let $P=MN$, $P'=M'N'$ be their corresponding parabolic subgroups. Assume that there exists $w\in W(G,A)$ such that $w(\theta)=\theta'$.

Given an irreducible admissible representation $\sigma$ of $M$, and $\nu\in \mathfrak{a}_{P,\mathbb{C}}^*$, where $\mathfrak{a}_{P,\mathbb{C}}^*=\mathfrak{a}_{P}^*\otimes_{\mathbb{R}}\mathbb{C}$,  in which $\mathfrak{a}_{P}^*=X^*(M)_F\otimes\mathbb{R}$, $X^*(M)_F$ is the lattice of $F$-rational characters of $M$, define the normalized parabolic induction as
$$I(\nu,\sigma)=\mathrm{Ind}_{P}^{G}(\sigma\otimes q_F^{\langle \nu+\rho_P, H_P(\cdot)\rangle})\otimes\textbf{1}_{N}$$
where $\rho_P$ is the half sum of positives roots in $N$, and $H_P:M(F)\rightarrow \mathfrak{a}_P=\Hom(X^*(M)_F,\mathbb{R})$ is the Harish-Chandra map characterized by $q_F^{\langle \chi, H_M(m)\rangle}=\vert \chi(m)\vert$ for all $\chi \in X^*(M)_F$.

Suppose  $\sigma$ is $\psi$-generic, i.e., there exists $0\neq \lambda\in \mathrm{Hom}_{U_M}(\sigma,\psi)$. Then the linear functional \begin{align*}\lambda_{\psi}(\nu,\sigma):I(\nu,\sigma)&\longrightarrow \mathbb{C}\\
f &\longmapsto \int_{N_1(F)}\lambda(f(\dot{w}_0^{-1} n'))\psi^{-1}(n')dn'
\end{align*}
defines a non-zero Whittaker functional on $I(\nu,\sigma)$, where $N_1=\dot{w}_0 \overline{N}\dot{w}_0^{-1}$, $\dot{w}_0:=\dot{w}_G\dot{w}_M^{-1}$, chosen to be compatible with $\psi$, i.e. $\psi(\dot{w}_0u\dot{w}_0^{-1})=\psi(u)$ for all $u\in U_M:=U\cap M$.
By Section 3.4 of \cite{Sha10}, the integral above converges as a principal value integral, and therefore holomorphic for all $\nu\in \mathfrak{a}_{P,\mathbb{C}}^*$.

One can also choose the representative $
\dot{w}$ of $w$ to be compatible with $\psi$, then the linear operator\begin{align*}A(\nu, \sigma,\dot{w}):I(\nu,\sigma)&\longrightarrow I(\dot{w}(\nu), \dot{w}(\sigma))\\
f&\longmapsto (g\mapsto \int_{N_{\dot{w}(F)}}f(\dot{w}^{-1}ng)dn)
 \end{align*} is $G$-equivariant, hence defines an intertwining operator between the induced representations. Here $N_{\dot{w}}=U \cap \dot{w}\overline{N}\dot{w}^{-1}$. As discussed in Section 4.1 of \cite{Sha10}, the integral converges absolutely for all $\nu$ in the region  $\mathrm{Re}\langle \nu, \alpha^
\vee\rangle\gg0$ for all simple roots $\alpha$ such that $w(\alpha)<0$. 
\subsubsection{Local coefficients} The composition of the Whittaker functional $\lambda_\psi(\dot{w}(\nu),\dot{w}(\sigma))$ on $I(\dot{w}(\nu), \dot{w}(\sigma))$ with the intertwining operator $A(\nu,\sigma,\dot{w})$ is another non-zero Whittaker functional on $I(\nu,\sigma)$. By uniqueness of local Whittaker functionals, there exists a non-zero constant $C_\psi(\nu,\sigma,\dot{w})$, called the Shahidi local coefficient, such that
$$\lambda_\psi(\nu,\sigma)=C_\psi(\nu,\sigma,\dot{w})\lambda_\psi(\dot{w}(\nu),\dot{w}(\sigma))\circ A(\nu,\sigma,\dot{w}).$$ By the holomorphy of p-adic Whittaker functionals and meromorphy of the intertwining operators, $C_\psi(\nu,\sigma,\dot{w})$ is a meromorphic function of $\nu$.
Theorem 4.2.2. of \cite{Sha10} provides an explicit decomposition of intertwining operators, which reduces the study of local coefficient to the case where $P$ is a maximal standard parabolic subgroup. Then $P=P_\theta$ where $\theta=\Delta-\{\alpha\}$ for some simple root $\alpha$. In this case, set $\nu=s\widetilde{\alpha}$, where $s\in \mathbb{C}$, $\widetilde{\alpha}=\langle \rho,\alpha\rangle^{-1}\rho$, $\rho$ is half sum of the positive roots in $N$. Denote by $C_\psi(s,\sigma):=C_\psi(s\widetilde{\alpha},\sigma, \dot{w}_0)$. Then Theorem 8.3.2, 2), \cite{Sha10} states that the local coefficient is a product of $\gamma$-factors: 
\[C_\psi(s,\sigma)=\prod_{i=1}^m\gamma(is,\sigma^\vee, r_i,\psi^{-1})\]
where $r_i$ is the irreducible constituent of the adjoint action $\mathrm{Ad}: {^L}M\rightarrow \mathrm{GL}({^L}\mathfrak{n})$ generated by the root vectors $X_{\beta^\vee}\in {^L}\mathfrak{n}$ subject to the condition $\langle \widetilde{\alpha},\beta^\vee\rangle=i$. From the global theory of the Langlands-Shahidi method, those $\gamma$-factors are exactly the ones associated to the L-functions appearing in the constant term of Eisenstein series.

We call a standard parabolic subgroup $P=P_\theta$ self-associate if there exists $w\in W(G,A)$ such that $w(\theta)=\theta$. When $P$ is self associate, it is proved in Theorem 6.2 of \cite{Sha02} that the local coefficient is a Mellin transform of certain partial Bessel functions under some assumptions. In this case, the integral representation of local coefficient provides a way to study the analytic properties of the $\gamma$-factors arising from the Langlands-Shahidi method. The analytic stability is one of the most important
properties that can be studied using this formula.
\section{Stability of exterior cube $\gamma$-factors}

\subsection{Some structural results of $E_6$}
\subsubsection{Root system of $E_6$}{\label{RSE6}}
Let $G=E^{\text{sc}}_6$ be the split simply connected group of type $E_6$ over a $p$-adic field $F$. Following Bourbaki's labeling \cite{Bou08}, the root system of $E_6$ has its Dynkin diagram as follows:

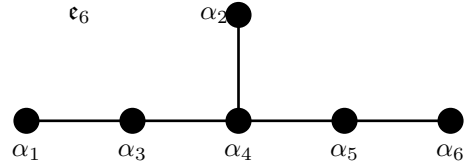
\begin{center}
   \begin{tikzpicture}[scale=.7]
    \foreach \x in {1,...,5}{
      \pgfmathparse{ 2*\x +3 };
      \pgfmathresult;
      \let\y\pgfmathresult;
      \fill[black] (\y , 2) circle (7pt) node[anchor=north] {};
    }
   \draw[color=black, line width=1pt] (5,2) -- (13,2);
   \draw[color=black, line width=1pt] (9,2) -- (9,4);
   \fill[black] (9,4) circle (7pt) node[anchor =east] {$\alpha_2$~~};
   \draw  (5,1.4) node[color= black] {$\alpha_1$};
   \foreach \x in {3,...,6}{
      \pgfmathparse{ 2*\x +1 };
      \pgfmathresult;
      \let\y\pgfmathresult;
      \draw  (\y,1.4) node[color= black] {$\alpha_{\x}$};
    }
   \draw (6,4) node[color=black] {$\mathfrak{e}_{6}$};
   \end{tikzpicture}
\end{center}
The set of roots is given by
$$\Phi=\{\pm e_i\pm e_j , \ \ (1\leq i< j \leq 5), \ \ \pm \frac{1}{2}(e_8-e_7-e_6+\sum_{i=1}^5 (-1)^{\nu(i)}e_i), \sum_{i=1}^5 \nu(i) \textrm{ even}\}$$

We have $\vert\Phi\vert=72$, $\vert\Phi^+\vert=36$. Fix a set of simple roots: 
$$\Delta=\{\alpha_1=\frac{1}{2}(e_1+e_8)-\frac{1}{2}(e_2+\cdots e_7), \alpha_2=e_1+e_2, \alpha_3=e_2-e_1, \alpha_4=e_3-e_2,$$$$ \alpha_5=e_4-e_3,\alpha_6=e_5-e_4 \}$$  To simplify notations, throughout this paper, we denote the root $a\alpha_1+b\alpha_2+c\alpha_3+d\alpha_4+e\alpha_5+f\alpha_6$ by $(abcdef)$. 
\subsubsection{The $(G,M)$-pair for the exterior cube representation}{\label{GMpair}} To study the exterior cube representation of $\mathrm{GL}_6$, we apply the Langlands-Shahidi method to the $(G,M)$-pair where $G=E^{\text{sc}}_6$, and $M$ is the Levi factor of the standard parabolic subgroup
$P=P_\theta$ of $G$ given by $\theta=\Delta-\{\alpha_2\}$. Let $P=MN$ be its Levi decomposition. Then
$\vert \Phi_M^+\vert=15$, $
\vert\Phi_N\vert=21$. Note that $N$ is a Heisenberg group with center $F$ given by the longest root $\gamma_0=(122321)$, and Witt index(dimension of maximal isotropic subspace) 10. We may take a basis of a maximal isotropic subspace of $N$ consisting of the  root vectors given by the following set of roots, partially ordered by the Bruhat order, i.e., $\alpha\ge\beta$ if $\alpha-\beta$ is a sum of non-negative roots:
\[E_0=\{(010000), (010100), (011100),(010110),(011110),(111100),(010111),(111110),(011111),(111111)\}\] For each root $\alpha\in E_0$, let $\alpha^*:=\gamma-\alpha$, and set
\[E_0^*=\{\alpha^*: \alpha\in E_0\}\]\[=\{(112321),(112221),(111221),(112211),(111211),(011221),(112210),(011211),(111210),(011210)\}\] Denote by $V_0$(resp. $V_0^*$) the span of the root vectors given by $E_0$(resp. $E_0^*$), and let $V=V_0\oplus V_0^*\oplus F$. Then $N\simeq H(V)$, the Heisenberg group associated to $V$.

Let $M_D$ be the derived group of $M$, then
We have $M_D\simeq \mathrm{SL}_6$. The split component $A_M$ of $M$ is
$$A_M:=(\cap_{\alpha\in \theta}\ker \alpha)^\circ=\{\gamma_0^\vee(t)=\alpha_1^\vee(t)\alpha_2^\vee(t^2)\alpha_3^\vee(t^2)\alpha_4^\vee(t^3)\alpha_5^\vee(t^2)\alpha_6^\vee(t):t\in F^\times\}.$$ Fix the following identification:
$$\alpha_1^\vee(t)\mapsto\mathrm{diag}\{t, t^{-1},1,1,1,1\}$$
$$\alpha_3^\vee(t)\mapsto\mathrm{diag}\{1,t, t^{-1},1,1,1\}$$
$$\alpha_4^\vee(t)\mapsto\mathrm{diag}\{1,1,t, t^{-1},1,1\}$$
$$\alpha_5^\vee(t)\mapsto\mathrm{diag}\{1,1,1,t, t^{-1},1\}$$
$$\alpha_6^\vee(t)\mapsto\mathrm{diag}\{1,1,1,1,t, t^{-1}\}$$We also
identify $\mathrm{GL}_1\simeq A_M$ by $t\mapsto \gamma_0^\vee(t)$. 
Since $2\alpha_2^\vee=\gamma_0^\vee-(\alpha_1^\vee+2\alpha_3^\vee+3\alpha_4^\vee+2\alpha_5^\vee+\alpha_6^\vee)$, we obtain an identification
$$A_M\times M_D\simeq \mathrm{GL}_1\times \mathrm{SL}_6$$ such that
$$\alpha_2^\vee(t^2)\mapsto (t, \mathrm{diag}\{t^{-1},t^{-1}, t^{-1},t,t,t\})$$
Consequently, $$A_M\cap M_D=\{\gamma_0^\vee(t): t^2=1\}$$ 
\[ M\simeq (\mathrm{GL}_1\times\mathrm{SL}_6)/S,\ \text{where } S=\{(t, tI_6): t^2=1\}.\] 
Define a map
$$\overline{\xi}: A_M\times M_D\rightarrow \mathrm{GL}_1\times \mathrm{GL}_1\times \mathrm{SL}_6 \mapsto \mathrm{GL}_1\times\mathrm{GL}_6$$
$$(\gamma_0^\vee(t), x)\mapsto (t^2, t, x)\mapsto (t^2, tx)$$
Observe that $\overline{\xi}$ factors though $A_M\cap M_D$, hence it induces a map
$$\xi: M\rightarrow \mathrm{GL}_1\times\mathrm{GL}_6$$ such that $\xi(\alpha_2^\vee(t))=(t, \mathrm{diag}\{1,1,1,t,t,t\})$, and we obtain a commutative diagram

\[ \begin{tikzcd}
A_M\times M_D \arrow{r}{\overline{\xi}} \arrow[swap]{d}{} & \mathrm{GL}_1\times \mathrm{GL}_6 \arrow{d}{\mathrm{Id}} \\%
M \arrow{r}{\xi}& \mathrm{GL}_1\times\mathrm{GL}_6
\end{tikzcd}
\]

It is particularly interesting to note that the image of $M$ under $\xi$ lies in the L-monoid for the exterior cube representation of $\mathrm{GL}_6$
$$\mathcal{M}^{\wedge^3}:=\{(a,g)\in \mathbb{A}^1\times \mathrm{Mat}_{6\times 6}: a^3=\det g\}.$$

For the general construction of L-monoids, we refer to \cite{Sha17}, section 7. It follows from the construction that the L-monoid comes with a generalized determinant map $\nu:\mathrm{GL}_1\times \mathrm{GL}_6\rightarrow \mathbb{G}_m$, such that $\nu(a,g)=a$.

\begin{lem}
    The following diagram commutes: 
    \[
  \begin{tikzcd}
    M \arrow{r}{\xi} \arrow[swap]{dr}{\gamma_0} & \mathrm{GL}_1\times \mathrm{GL}_6 \arrow{d}{\nu} \\
     & \mathbb{G}_m
  \end{tikzcd}
\]
\end{lem}
\begin{proof}
Note that both $\gamma_0$  and $\nu\circ \xi$
are morphisms of algebraic groups and factor through
\[
M/M_{D}\simeq A_M/S.
\]
As a consequence, it suffices to check that they coincide on $(\gamma_0^\vee(t), 1)\in M$.
Observe that 
\[
\nu(\xi(\gamma_0^\vee(t)))=\nu((t^2, t))=t^2
\]
and $\gamma_0(\gamma_0^\vee(t))=t^{\langle \gamma_0^\vee,\gamma_0\rangle}=t^2$.

\end{proof}
Set $w_0=w_Gw_M^{-1}$, then $w_0(\theta)=\theta$ and $w_0(\alpha_2)<0$. It follows that $P=MN$ is self-associate.  For $i=1,2\cdots, 6$, denote by $s_i=s_{\alpha_i}$, the simple reflection associated to $\alpha_i$. Then $w_G$ and $w_M$ admit the following reduced decomposition: 
\[w_G=
s_1  s_2  s_3  s_1  s_4  s_2  s_3  s_1  s_4  s_3  s_5  s_4  s_2  s_3  s_1  s_4  s_3  s_5  s_4  s_2  s_6  s_5  s_4  s_2  s_3  s_1  s_4  s_3  s_5  s_4  s_2  s_6  s_5  s_4  s_3  s_1
\]
 \[w_M=
 s_1  s_3  s_1  s_4  s_3  s_1  s_5  s_4  s_3  s_1  s_6  s_5  s_4  s_3  s_1
 \]
As a result, the relative long Weyl group element $w_0=w_G\cdot w_M^{-1}$ has the following decomposition:
\begin{equation}\label{eqn: simple-reflection-w0}
w_0=
s_{2}  s_{4}  s_{3}  s_{1}  s_{5}  s_{4}  s_{2}  s_{3}  s_{4}  s_{5}  s_{6}  s_{5}  s_{4}  s_{2}  s_{3}  s_{1}  s_{4}  s_{3}  s_{5}  s_{4} s_{2}
\end{equation}

According to Remark 5.1.3 of \cite{Sha10}, it is important to choose the representative $\dot{s}_i:=\xi_i(\begin{pmatrix}
     & 1\\
     -1 & 
\end{pmatrix})$ for each simple root $\alpha_i$, where $\xi_i: \mathrm{SL}_2\rightarrow G$ is the homomorphism determined by the root $\alpha_i$ upon a fixed vector $0\neq X_{\alpha_i}\in\mathfrak{u}_{\alpha_i}$, $\mathfrak{u}_{\alpha_i}=\mathrm{Lie}(U_{\alpha_i})$, as in the Jacobson-Morosov theorem. Let $\dot{w}_G$, $\dot{w}_M$, and $\dot{w}_0$ be chosen as the product of the $\dot{s}_i$'s in their reduced decompositions. Then $\dot{w}_0=\dot{w}_G\dot{w}_M^{-1}$.
One checks that 
\begin{lem}\label{w_0-action-on-roots}$\dot{w}_0(\alpha_i)=\alpha_i$ for all $i\neq 2$, \[\dot{w}_0(\alpha_2)=-(112321)=-(\alpha_1+\alpha_2+2\alpha_3+3\alpha_4+2\alpha_5+\alpha_6), \mathrm{\ and\ }\dot{w}_0 x_{\alpha_i}(a_i)\dot{w}_0^{-1}=x_{\alpha_i}(a_i),\  \forall a_i\in \mathbb{G}_a.\]
\end{lem}

We also need the following lemma:
\begin{lem}{\label{gamma_0(w_M)}}
We have $\gamma_0(\dot w_M)=1$.
\end{lem}

\begin{proof} Let $(\epsilon,\tilde{w}_M)\in A_M\times M_D$ be any lift of $\dot{w}_M$, then  $\gamma_0(\dot{w}_M)=\nu(\overline{\xi}(\epsilon,\widetilde{w}_M))=\epsilon^2$. Note that $\epsilon^2$ is also the $A_M$-component of the corresponding lift of $\dot{w}_M^2$. On the other hand, from the reduced decomposition of $\dot{w}_M$, we have $\dot{w}_M^2=\alpha_1^\vee(-1)\alpha_4^\vee(-1)\alpha_6^\vee(-1)$. We obtain that $\epsilon^2=1$.
    
\end{proof}
Next, let us explain how to obtain the local $\gamma$-factor for the exterior cube representation of $\mathrm{GL}_6$ using the $(G,M)$-pair in $E_6$ via the local coefficients in the context of the Langlands-Shahidi method. Fix a non-trivial additive character $\psi$ of $F$. Together with the splitting given by $x_\beta: \mathbb{G}_a\rightarrow U_\beta,\ \beta\in \Delta$, it defines a generic character of $U$, still denoted as $\psi$. By Lemma \ref{w_0-action-on-roots}, $\psi$ is compatible with $\dot{w}_0$ in the sense that $\psi(\dot{w}_0u\dot{w}_0^{-1})=\psi(u)$ for all $u\in U_M$. Let $\pi$ be a supercuspidal (hence $\psi$-generic) representation of $\mathrm{GL}_6(F)$ with central character $\omega_\pi$. By tensoring with the trivial character on the $\mathrm{GL}_1$-component and pull back along $\xi$, we obtain a representation $$\sigma:=\xi^*(1\otimes \pi)$$ of $M(F)$. Consequently, $\omega_\sigma=\omega_\pi\otimes 1$. The adjoint action $\mathrm{Ad}: {^LM}\rightarrow \mathrm{GL}({^L\mathfrak{n}})$ breaks up into two irreducible components $r_1\simeq \wedge^3\mathbb{C}^6$ and $r_2$ is the one-dimensional representation given by the highest root vector $\gamma$. By Theorem 8.3.2 of \cite{Sha10}, and section 2.5.3 of \cite{Kim2005OnLL}, we have
$$C_\psi(s,\sigma)=\gamma(s, \pi^\vee, \wedge^3,\psi^{-1})\gamma(2s,\omega_{\pi^\vee},\psi^{-1})$$
where $\pi^\vee$ is the contragredient of $\pi$ and $C_\psi(s,\sigma)$ is the local coefficient.



\subsubsection{$U_M$-orbits on $N'$} When the parabolic subgroup $P=MN$ is self-associate, the local coefficient can be written as the Mellin transform of certain partial Bessel functions by Theorem 6.2 of \cite{Sha02}. The space of integration is the geometric quotient $U_M\backslash N'$ where $U_M=U\cap M$ and $N'\subset N$ is some Zariski open dense subset, with $U_M$ acting on $N$ by conjugation. We begin by studying the structure of the geometric quotient $U_M\backslash N'$. 
\begin{thm}\label{orbit-space-representative}
Let
\[
\gamma_0=(122321),\quad \gamma_1=(112321),\quad
\gamma_2=(112221),\quad \gamma_3=(111211),\quad
\gamma_4=(011210),\quad \gamma_5=(010000),
\]
and set
\[
S=\left\{\dot n(c)=\prod_{i=0}^{5}x_{\gamma_i}(c_i):
(c_0,\ldots,c_5)\in \mathbb{A}^6\right\}\subset N,
\]
where the product is taken in the displayed order. There exists a non-empty
$F$-rational Zariski open subset $S_0\subset S$ and a saturated Zariski open subset $N'\subset N$ such that the action morphism
\[
\mu:U_M\times S_0\longrightarrow N',
\qquad (u,s)\longmapsto usu^{-1},
\]
is an isomorphism of $F$-varieties. Consequently $U_M\backslash N'\simeq S_0$
is a geometric quotient, $U_M$ acts freely on $N'$, and for every field
extension $E/F$ every $U_M(E)$-orbit in $N'(E)$ contains a unique point of
$S_0(E)$. In particular, the $F$-rational orbits in $N'(F)$ are represented
uniquely by the points
\[
\prod_{i=0}^{5}x_{\gamma_i}(c_i),\qquad c_i\in F^\times.
\]
The open subset may be further shrunk without changing these conclusions.
\end{thm}

\begin{proof}
Since $F$ has characteristic zero, the exponential map gives
an $F$-variety isomorphism from the Lie algebra of every unipotent group
appearing below to the group itself. We use it only to identify the action:
the center of $N$ is the root subgroup $U_{\gamma_0}$, and
\[
V:=N/Z(N)\simeq \wedge^3F^6
\]
as a $U_M$-module under the identification
$M_{D}\simeq \mathrm{SL}_6$. We normalize the 
coordinates by
\[
X_{\gamma_1}=e_{123},\qquad X_{\gamma_2}=e_{124},\qquad
X_{\gamma_3}=e_{135},\qquad X_{\gamma_4}=e_{236},\qquad
X_{\gamma_5}=e_{456}.
\]

Let $x_{ijk}$ denote the coordinate function dual to $e_{ijk}$ on
$\wedge^3F^6$. Let $D_r$ be the  operator which sends $e_{r+1}$ to $e_r$
($1\le r\le5$), then on coordinate functions
\[
D_r(x_I)=
\begin{cases}
\epsilon(I,r)\,x_{(I\setminus\{r\})\cup\{r+1\}},
& r\in I,\ r+1\notin I,\\
0,&\text{otherwise},
\end{cases}
\]
where $I=\{i<j<k\}$ and $\epsilon(I,r)=\pm1$ is the sign needed to reorder
the wedge. Thus
\[
F[V]^{U_M}=\bigcap_{r=1}^5\ker D_r .
\]
Brion's computation \cite[Theorem 3 and Table on Page 13]{Brion83} says that
this invariant algebra is polynomial:
\[
F[V]^{U_M}=F[f_1,f_2,f_3,f_4,f_5],
\]
with degrees and $\mathrm{SL}_6$-highest weights
\[
(1,\omega_3),\quad (2,\omega_1+\omega_5),\quad
(3,\omega_3),\quad (4,\omega_2+\omega_4),\quad (4,0).
\]
Although Brion states the result over $\mathbb C$, the same generators and
relations are valid over $F$.

Let $\iota:S\hookrightarrow N$ denote the inclusion. For the
functions $f_i$ on $V=N/Z(N)$, $\iota^*f_i$ is the pullback along the
composition $S\hookrightarrow N\twoheadrightarrow V$.

We also record the finite calculation which fixes $\iota^*f_i$.
Starting with a general homogeneous polynomial of the indicated degree and
weight and imposing $D_r(f)=0$ for $1\le r\le5$ gives the following
normalization of Brion's generators. When restricted to the affine subspace
with coordinates
\[
x_{123}=c_1,\quad x_{124}=c_2,\quad x_{135}=c_3,\quad
x_{236}=c_4,\quad x_{456}=c_5,
\]
all other $x_{ijk}$ being zero, their restrictions are
\begin{align}
\iota^*f_1 &= c_5,\label{eq:slice-f1}\\
\iota^*f_2 &= c_4c_5,\label{eq:slice-f2}\\
\iota^*f_3 &= c_1c_5^2,\label{eq:slice-f3}\\
\iota^*f_4 &= c_3c_4c_5^2,\label{eq:slice-f4}\\
\iota^*f_5 &= 4c_2c_3c_4c_5+c_1^2c_5^2.\label{eq:slice-f5}
\end{align}
For example, in degree four and weight zero the restriction to this subspace
can only involve the two monomials $c_2c_3c_4c_5$ and $c_1^2c_5^2$, and the
kernel calculation above gives the displayed coefficients. 

Let $z$ be the linear coordinate on the central root space
$\mathfrak g_{\gamma_0}$ in exponential coordinates. Since $U_M$ fixes the
highest root vector $X_{\gamma_0}$, it acts trivially on this coordinate, and
\[
F[N]^{U_M}=F[z]\otimes_F F[V]^{U_M}=F[z,f_1,\ldots,f_5].
\]
In the ordered root-subgroup coordinates used in the statement, the
exponential central coordinate differs from the displayed $x_{\gamma_0}$-
coordinate only by the central Baker--Campbell--Hausdorff term coming from
$x_{\gamma_1}$ and $x_{\gamma_5}$; hence, on $S$,
\[
\iota^*z=c_0+\lambda c_1c_5
\]
for a constant $\lambda\in F$ depending only on the fixed Chevalley
normalization and on the chosen ordering.

Let
\[
Q=\operatorname{Spec}F[N]^{U_M},\qquad q:N\longrightarrow Q
\]
be the categorical quotient morphism, and put $Q_0=D(f_1f_2f_4)\subset Q$.
Let $S^\ast\subset S$ be the open subset defined by $c_3c_4c_5\ne0$. By the
displayed restriction formulas, $q|_{S^\ast}:S^\ast\to Q_0$ is an
isomorphism, with inverse
\begin{align*}
c_5 &= f_1, &
c_4 &= \frac{f_2}{f_1}, &
c_1 &= \frac{f_3}{f_1^2},\\
c_3 &= \frac{f_4}{f_1f_2}, &
c_2 &= \frac{f_1^2f_5-f_3^2}{4 f_1f_4}, &
c_0 &= z-\lambda\frac{f_3}{f_1}.
\end{align*}
This proves that the proposed slice maps isomorphically to the open
categorical quotient. It remains to prove that the quotient is geometric on
the corresponding open subset of $N$.

Let
\[
\mathcal B=\Phi_N\setminus\{\gamma_0,\gamma_1,\ldots,\gamma_5\}.
\]
These are the fifteen $N$-root coordinates not present on the slice, and
$15=\dim U_M$. The table below should be read as follows. For each
$\beta\in\mathcal B$ there is a positive root $\delta_\beta$ of $M$ and one
of the roots $\gamma_i$ such that
\[
\beta=\gamma_i+\delta_\beta .
\]
The third column records the corresponding  coordinate $c_i$. As
$\beta$ runs through $\mathcal B$, the roots $\delta_\beta$ run through
$\Phi_M^+$ exactly once:
\[
\begin{array}{c|c|c}
\beta & \delta_\beta & \text{pivot coordinate} \\ \hline
(010100)&(000100)&c_5\\
(011100)&(001100)&c_5\\
(010110)&(000110)&c_5\\
(011110)&(001110)&c_5\\
(111100)&(101100)&c_5\\
(010111)&(000111)&c_5\\
(111110)&(101110)&c_5\\
(011111)&(001111)&c_5\\
(111111)&(101111)&c_5\\
(111221)&(000010)&c_3\\
(112211)&(001000)&c_3\\
(011221)&(000011)&c_4\\
(112210)&(101000)&c_4\\
(011211)&(000001)&c_4\\
(111210)&(100000)&c_4
\end{array}
\]
Order $\Phi_M^+$, and therefore also $\mathcal B$, by increasing height of
$\delta_\beta$, with any fixed order among roots of equal height. Write
\[
u=\prod_{\delta\in\Phi_M^+}^{\rightarrow}x_\delta(u_\delta),
\qquad
s=\prod_{i=0}^5x_{\gamma_i}(c_i),
\]
and let $y_\eta$ be the ordered root-subgroup coordinate on $N$ attached to
$\eta\in\Phi_N$. For the action map
\[
\mu:U_M\times S^\ast\longrightarrow N,\qquad (u,s)\longmapsto usu^{-1},
\]
the Chevalley commutator formula gives, for each $\beta\in\mathcal B$,
\[
\mu^*(y_\beta)
=\varepsilon_\beta p_\beta u_{\delta_\beta}
+P_\beta\bigl(c_0,\ldots,c_5;
u_{\delta'}:\operatorname{ht}(\delta')<\operatorname{ht}(\delta_\beta)\bigr),
\qquad \varepsilon_\beta=\pm1,
\]
where $p_\beta$ is the pivot coordinate from the table. The reason is simple:
the displayed linear term comes from
\[
x_{\delta_\beta}(u_{\delta_\beta})x_{\gamma_i}(c_i)
x_{\delta_\beta}(u_{\delta_\beta})^{-1}
=x_{\gamma_i}(c_i)x_\beta(\varepsilon_\beta c_i u_{\delta_\beta})
\cdot(\text{higher root factors}),
\]
and every other contribution to the same $\beta$-coordinate involves only
earlier $u_{\delta'}$'s in the height order. 
More precisely, if
$\mathcal B=\{\beta_1,\ldots,\beta_{15}\}$ is ordered as above, then the
$j$-th equation has the form
\[
\mu^*(y_{\beta_j})
=\varepsilon_{\beta_j}p_{\beta_j}u_{\delta_{\beta_j}}
+P_{\beta_j}(c_0,\ldots,c_5;
u_{\delta_{\beta_1}},\ldots,u_{\delta_{\beta_{j-1}}}).
\]
In particular, up to sign we have
$
\prod_{j}p_{\beta_j}=c_3^2c_4^4c_5^9,
$
which is invertible on $S^\ast$.

We now solve the orbit problem explicitly. Take $n\in q^{-1}(Q_0)$. Take $n\in q^{-1}(Q_0)$.
By definition, $Q_0=D(f_1f_2f_4)$ is the locus in $Q$ where the product
$f_1f_2f_4$ is nonzero.
 Define
$s(c)=\prod_{i=0}^{5}x_{\gamma_i}(c_i)\in S^*$
by taking
\begin{align*}
c_5 &= f_1(n), &
c_4 &= \frac{f_2(n)}{f_1(n)}, &
c_1 &= \frac{f_3(n)}{f_1(n)^2},\\
c_3 &= \frac{f_4(n)}{f_1(n)f_2(n)}, &
c_2 &= \frac{f_1(n)^2f_5(n)-f_3(n)^2}{4 f_1(n)f_4(n)}, &
c_0 &= z(n)-\lambda\frac{f_3(n)}{f_1(n)} .
\end{align*}
 Substituting
the definitions of the $c_i$ into
\eqref{eq:slice-f1}--\eqref{eq:slice-f5} and into
$\iota^*z=c_0+\lambda c_1c_5$ gives
\[
f_i(s(c))=f_i(n)\quad(1\le i\le5),\qquad z(s(c))=z(n).
\]
Since $Q=\operatorname{Spec}F[z,f_1,\ldots,f_5]$, this is exactly the
statement that $q(s(c))=q(n)$. Conversely, if $s(c')\in S^\ast$ has
$q(s(c'))=q(n)$, the same formulas force $c'=c$ successively. Hence $s(c)$
is the unique point of $S^\ast$ lying over $q(n)$. Then the equations
\[
y_\beta(n)=\mu^*(y_\beta)(u,c),\qquad \beta\in\mathcal B,
\]
determine the coordinates $u_{\delta_\beta}$ recursively and uniquely. Since
the recursion only divides by the pivot coordinates $c_3,c_4,c_5$, the
solution depends regularly on $n$ over $q^{-1}(Q_0)$. Hence
\[
\mu:U_M\times S^\ast\longrightarrow q^{-1}(Q_0)
\]
has a regular inverse and is an isomorphism of $F$-varieties.

\smallskip
Finally set
\[
S_0=S^\ast\cap D(c_0c_1c_2),
\]
or replace it by any smaller non-empty open subset needed later, and define
$N'=q^{-1}(q(S_0))$. The isomorphism above restricts to
\[
\mu:U_M\times S_0\xrightarrow{\sim}N' .
\]
Under this identification the quotient map is the projection
$U_M\times S_0\to S_0$, so $U_M\backslash N'\simeq S_0$ is a geometric
quotient. The same product decomposition also proves freeness: if
$usu^{-1}=s$ with $s\in S_0$, uniqueness of the factorization gives $u=1$.
All constructions are defined over $F$, hence after any field extension
$E/F$ we still have
\[
N'(E)=U_M(E)\cdot S_0(E)
\]
with unique factorization. This gives the asserted unique representatives;
because $S_0\subset D(c_0c_1c_2c_3c_4c_5)$, they have $c_i\in F^\times$.
\end{proof}

\begin{remark}
In the applications below we replace $S_0$ by smaller open subsets by imposing the
additional Bruhat and \'etaleness conditions. We continue to denote the
corresponding saturated open subset of $N$ by $N'$.
\end{remark}


\subsubsection{Invariant measure}Theorem 6.2 of \cite{Sha02} states that the local coefficient is the Mellin transform of certain partial Bessel functions over the p-adic manifold obtained by the $F$-points of the geometric quotient $U_M\backslash N'$. In general, the existence of invariant measures on $U_M\backslash N'$ is known \cite[Theorem 2.51]{Foll16}. But for the proof of stability, an explicit form of the invariant measure is crucial. 
We consider the action 
\[
\Phi: U_M\times F^6\rightarrow N, \quad (u, \dot{n})\mapsto u\dot{n}u^{-1}.
\] where we identify $F^6$ with $\prod_{i=0}^5 U_{\gamma_i}(F)$ through the splitting $x_{\gamma_i}: \mathbb{G}_a\rightarrow U_{\gamma_i}, \ (i=0,\cdots, 5)$.
\begin{prop}\label{Invmeas}Let $dn$(resp. $du$) be the Haar measure on $N(F)$(resp. $U_M(F)$), normalized so that $\mathrm{Vol}(N(\mathcal{O}_F),dn)=1$(resp.  $\mathrm{Vol}(U_M(\mathcal{O}_F),du)=1$ ).
We have 
\[
\Phi^{*}(dn)=J_{\Phi}(c_0,\ldots, c_5)du\wedge \prod_{i=0}^5dc_i
\]
where 
\[
J_{\Phi}(c_0,\ldots, c_5)=|c_3|^2|c_4|^4|c_5|^9,
\]
\end{prop}

\begin{remark}
The exponents are determined as follows: 
Observe that for any $\gamma\in (E_0\cup E_0^*)\backslash R$, there is a maximal $\gamma_i\in R$ such that 
$\gamma-\gamma_i>0$, i.e., a positive linear combination of simple roots. We write
\[
\varepsilon_i=|\{\gamma\in (E_0\cup E_0^*)\backslash R: \gamma-\gamma_i>0\}|.
\]
Then $J_{\Phi}(c_0,\ldots, c_5)=\prod_{i=0}^5|c_i|^{\varepsilon_i}$. 
\end{remark}

\begin{proof}Let $J_\Phi(u,c)$ be the Jacobian of $\Phi$, and $dc=\prod_{i=0}^5 dc_i$, in which $dc_i$ is the restriction of the Haar measure of $dn$ to the root subgroup $U_{\gamma_i}(F)$. Then
$$\Phi^* dn=du\wedge J_\Phi(u,c)dc$$ We will first show that $J_\Phi(u,c)$ is independent of $u\in U_M$. Denote the map $u\mapsto u'u$ by $L_{u'}$, where $u'\in U_M$. Then we have the following commutative diagram:

\[ \begin{tikzcd}
U_M\times F^6 \arrow{r}{\Phi} \arrow[swap]{d}{L_{u'}\times Id_{F^6}} & N \arrow{d}{Ad(u')} \\%
U_M\times F^6 \arrow{r}{\Phi}& N
\end{tikzcd}
\]
Then $$du\wedge J_\Phi(u,c)dc=\Phi^* dn=\Phi^*Ad(u')^*dn=(L_{u'}\times Id_{F^6})^*\Phi^* dn=L_{u'}^*du\wedge J(u'u,c)dc=du\wedge J(u'u,c)dc$$ since $dn$ is $Ad(u')$-invariant, and $du$ is $L_{u'}$-invariant. Thus we conclude that $$J_\Phi(u,c)=J_\Phi(u'u,c)$$ for all $u'\in U_M$. This shows that $J_\Phi(u,c)$ is independent of $u\in U_M$, and we denote it by $J_\Phi(c)$.

Note that $\Phi$ is $A(F)$-equivariant. As $J_\Phi\in \mathcal{O}(\mathbb{A}^6)$ is a rational function of the $c_i$'s, and $t u_\beta(x)t^{-1}=u_\beta(\beta(t)x)$ for all $\beta \in \Phi_G$, $t\in A(F)$, we see that
$$Ad(t)^* J_\Phi=\vert t\vert^\lambda J_\Phi$$ for some power $\lambda$ depending only on the roots $\gamma_i(i=0,\cdots,5)$. Consequently, $J_\Phi$ is a monomial of $\vert c_0\vert,\cdots, \vert c_5\vert$, i.e. $J_\varphi(c)=\prod_{i=0}^5\vert c_i\vert^{k_i}$, $k_i\in \mathbb{Z}$. 

The left part is to determine the $k_i(i=0,\cdots, 5)$. By the $A(F)$-equivariance of $\Phi$, we have
$$\Phi^* Ad(t)^*dn=Ad(t)^*du\wedge J_\Phi(Ad(t)c) Ad(t)^*dc$$ In particular, replace $t$ by $\gamma_i^\vee(t)$ for $i=0,\cdots, 5$ in the above formula, we obtain a system of linear equations:  
$$\sum_{\beta\in \Phi_N^+}\langle \beta,\gamma_i^\vee\rangle-\sum_{\beta\in \Phi_M^+}\langle\beta, \gamma_i^\vee\rangle-\sum_{j=0}^5\langle\gamma_j, \gamma_i^\vee\rangle( k_j+1)=0$$ where $\Phi_N^+$ and $\Phi_M^+$ are positive roots in $N$ and $M$ respectively. Note that we have
$$2\rho=\sum_{\beta\in \Phi_N^+}\beta=11(\alpha_1+2\alpha_2+2\alpha_3+3\alpha_4+2\alpha_5+\alpha_6)=11\gamma_0,\ \sum_{\beta\in \Phi_M^+}\beta=5\alpha_1+8\alpha_3+9\alpha_4+8\alpha_5+5\alpha_6$$
\begin{itemize}
    \item $\langle\alpha_k, \gamma_0^\vee\rangle=0$ for $k=1,3, 4,5,6 $ and $\langle \alpha_2,\gamma_0^\vee\rangle=1$, thus $\langle\gamma_k,\gamma_0^\vee\rangle=1$ for $i=1,2,3,4,5$ and $\langle \gamma_0,\gamma_0^\vee\rangle=2$; 
    \item $\langle \alpha_k,\gamma_1^\vee\rangle=0$ for $k=1,3,5,6$, $\langle \alpha_2,\gamma_1^\vee\rangle=-1,\langle\alpha_4,\gamma_1^\vee\rangle=1$, thus $\langle \gamma_k, \gamma_1^\vee\rangle=1$ for $k=0,2,3,4$, and $\langle \gamma_1,\gamma_1^\vee\rangle=2$, $\langle \gamma_5,\gamma_1^\vee\rangle=-1$;
    \item $\langle \alpha_k,\gamma_2^\vee\rangle=0$ for $k=1,2,6$, $\langle \alpha_3,\gamma_2^\vee\rangle=1,\langle\alpha_4,\gamma_2^\vee\rangle=-1,\langle\alpha_5,\gamma_2^\vee\rangle=1$, thus $\langle \gamma_k,\gamma_2^\vee\rangle=0$ for $k=3,4,5$, and $\langle\gamma_0,\gamma_2^\vee\rangle=1$,$\langle\gamma_1,\gamma_2^\vee\rangle=1$, $\langle\gamma_2,\gamma_2^\vee\rangle=2$;
    \item $\langle \alpha_k,\gamma_3^\vee\rangle=1$ for $k=1, 4,6$, $\langle\alpha_3,\gamma_3^\vee\rangle=-1$, $\langle\alpha_4,\gamma_3^\vee\rangle=1$, $\langle \alpha_5,\gamma_3^\vee\rangle=-1$, thus
    $\langle \gamma_k, \gamma_3^\vee\rangle=0$ for $k=2,4,5$, and $\langle\gamma_0,\gamma_3^\vee\rangle=1$, $\langle\gamma_1,\gamma_3^\vee\rangle=1$, $\langle\gamma_3,\gamma_3^\vee\rangle=2$;
    \item $\langle \alpha_k,\gamma_4^\vee\rangle=1$ for $k=2,3,5$, $\langle\alpha_1,\gamma_4^\vee\rangle=-1$, $\langle\alpha_4,\gamma_4^\vee\rangle=1$, $\langle \alpha_6,\gamma_4^\vee\rangle=-1$, thus
    $\langle \gamma_k, \gamma_4^\vee\rangle=0$ for $k=2,3,5$, and $\langle\gamma_0,\gamma_4^\vee\rangle=1$, $\langle\gamma_1,\gamma_4^\vee\rangle=1$, $\langle\gamma_4,\gamma_4^\vee\rangle=2$;
    \item $\langle \alpha_k,\gamma_5^\vee\rangle=0$ for $k=1,3,5,6$, $\langle\alpha_2,\gamma_5^\vee\rangle=2$, $\langle\alpha_4,\gamma_5^\vee\rangle=-1$, thus
    $\langle \gamma_k, \gamma_5^\vee\rangle=0$ for $k=2,3,4$, and $\langle\gamma_0,\gamma_5^\vee\rangle=1$, $\langle\gamma_1,\gamma_5^\vee\rangle=-1$, $\langle\gamma_5,\gamma_5^\vee\rangle=2$.  
\end{itemize} 

From the above calculations we obtain the following system of linear equations:
$$22-2(k_0+1)-\sum_{j=1}^5(k_j+1)=0;$$
$$11-9-2(k_1+1)+(k_5+1)-\sum_{j= 0,2,3,4}(k_j+1)=0;$$
$$11-7-(k_0+1)-(k_1+1)-2(k_2+1)=0;$$
$$11-3-(k_0+1)-(k_1+1)-2(k_3+1)=0;$$
$$11-(-1)-(k_0+1)-(k_1+1)-2(k_4+1)=0;$$ 
$$11-(-9)-(k_0+1)+(k_1+1)-2(k_5+1)=0.$$
Let $x_i=k_i+1$, we get 
$$\begin{pmatrix}
    2 & 1 & 1 & 1 & 1 & 1\\
    1 & 2 & 1 & 1 & 1 & -1\\
    1 & 1 & 2 & 0 & 0 & 0\\
    1 & 1 & 0 & 2 & 0 & 0\\
    1 & 1 & 0 & 0 & 2 & 0\\
    1 & -1 & 0 & 0 & 0 & 2
\end{pmatrix}\begin{pmatrix}
    x_0\\
    x_1\\
    x_2\\
    x_3\\
    x_4\\
    x_5
\end{pmatrix}=\begin{pmatrix}
    22\\
    2\\
    4\\
    8\\
    12\\
    20\\
\end{pmatrix}$$ which admits a unique solution $(x_0,x_1,x_2,x_3, x_4,x_5)=(1,1,1,3,5,10)$, hence $$(k_0,k_1,k_2,k_3,k_4,k_5)=(0,0,0,2,4,9),$$ as claimed in the theorem.
\end{proof}

\begin{remark}
The same proof should work for other cases when the action of $U_M$ on $N$ has trivial stabilizers.
\end{remark}

\begin{cor} The space of $F$-points of the orbit space $U_M\backslash N'$, identified with a dense subset of $(F^\times)^6$ through $\Xi$, admits a $U_M$-invariant measure of the form  $$|c_3|^2|c_4|^4|c_5|^9\prod_{i=0}^5dc_i$$
\end{cor}

\subsection{A Bruhat decomposition} Since $P\overline{N}$ is open dense in $G$, there exists an open dense subset $N''\subset N$ such that the Bruhat decomposition
$$\dot{w}_0^{-1}n=mn'\overline{n} $$ holds
for $n\in N''$, $m\in M$, $n'\in N$, and $\overline{n}\in \overline{N}$. Shrink $N'$ if necessary, we may assume that the above decomposition and the geometric quotient $U_M\backslash N'$ both exist for $n\in N'$. Although it is complicated to characterize $N'$ and the map $n\mapsto m$ by the decomposition explicitly for generic $n\in N'$, we have an explicit description for representatives in the orbit space $U_M\backslash N'$.

\begin{prop}{\label{BruhatDecomp}}The map $\mathcal{N}: N'\rightarrow M$ given by $\dot{w}_0^{-1}n=mn'\overline{n}$ is 
$U_M$-equivariant. Suppose $n=\prod_{i=0}^5x_{\gamma_i}(c_i)$, $c_i\in F^\times(i=0,\cdots 5)$ is a typical representative for $U_M\backslash N'$. Assume that
$c_0-c_1c_5$, $c_0(c_0-c_1c_5)-c_2c_3c_4c_5\in F^\times$. Then the Bruhat decomposition $\dot{w}_0^{-1}n=mn'\overline{n}$ holds with $m$ lying in the big Bruhat cell of $M$. Write $m=u_1 t\dot{w}_M u_2$, then
\begin{align*}
u_1&=x_{\alpha_{23}}((c_0 - c_1c_5)c_4^{-1}c_5^{-1})x_{\alpha_{15}}((c_0 - c_1c_5)c_3^{-1}c_5^{-1})x_{\alpha_{4}}((-c_0 + c_1c_5)c_2^{-1}c_5^{-1}),\\
t&=\alpha_1^\vee(-c_4c_5)\alpha_2^\vee(c_0(c_0-c_1c_5)-c_2c_3c_4c_5)\alpha_3^\vee(-c_3c_4c_5^2)\alpha_4^\vee(-c_2c_3c_4c_5^3)\alpha_5^\vee(-c_3c_4c_5^2)\alpha_6^\vee(-c_4c_5)\\
u_2&=x_{\alpha_{23}}(c_0c_4^{-1}c_5^{-1})x_{\alpha_{15}}(c_0c_3^{-1}c_5^{-1})x_{\alpha_{4}}(-c_0c_2^{-1}c_5^{-1}),
\end{align*} in which 
\[
\alpha_{23}=(101111), \alpha_{15}=(001110),
\alpha_{4}=(000100).
\] Moreover, the coefficient of $x_\alpha$ in $\dot{w}_0^{-1}\overline{n}\dot{w}_0$ is given by
\[
\frac{-c_0c_5 + c_1c_5^2}{c_0^2 - c_0c_1c_5 - c_2c_3c_4c_5}.
\]
\end{prop}

The explicit Bruhat decomposition and the formula for $\mathcal{N}$ were first found by a computation in Magma \cite{Magma97}. Here we sketch a proof for completeness. 

\begin{proof}
The equivariance is formal and is independent of the coordinate calculation.
Suppose
\[
\dot{w}_0^{-1}n=mn'\overline{n},
\qquad m\in M,\quad n'\in N,\quad \overline{n}\in \overline{N}.
\]
For $u\in U_M$, since $M$ normalizes both $N$ and $\overline{N}$, we have
\begin{align*}
\dot{w}_0^{-1}(unu^{-1})
&=(\dot{w}_0^{-1}u\dot{w}_0)(\dot{w}_0^{-1}n)u^{-1}\\
&=\Theta_M(u)mn'\overline{n}u^{-1}\\
&=(\Theta_M(u)mu^{-1})(un'u^{-1})(u\overline{n}u^{-1}),
\end{align*}
where $\Theta_M=\mathrm{Ad}(\dot{w}_0^{-1})$. By uniqueness of the
factorization in the open cell $P\overline{N}=MN\overline{N}$, the
$M$-component is $\Theta_M(u)mu^{-1}$. This proves the stated equivariance.

It remains to justify the displayed formula for $m$ and the coefficient of
$x_\alpha$. Set
\[
X=c_0-c_1c_5,\qquad Y=c_0(c_0-c_1c_5)-c_2c_3c_4c_5 .
\]
We make the calculation as an identity over the localized ring
\[
R_0=\mathbb{Z}
\bigl[c_0^{\pm1},c_1^{\pm1},c_2^{\pm1},c_3^{\pm1},
c_4^{\pm1},c_5^{\pm1},X^{-1},Y^{-1}\bigr].
\]
After base change from $R_0$ to $F$, this is exactly the open subset specified
in the proposition. Thus it is enough to prove the identity in the split
Chevalley group scheme $G_{R_0}$ attached to the fixed pinning.

We use
\[
\dot{s}_i=x_{\alpha_i}(1)x_{-\alpha_i}(-1)x_{\alpha_i}(1),
\]
and the reduced words for $\dot{w}_0$ and $\dot{w}_M$ fixed above. The only
relations used in the calculation are the following Chevalley identities:
\[
\dot{s}_i^{-1}x_\beta(r)\dot{s}_i
=x_{s_i\beta}(\varepsilon_i(\beta)r),
\qquad \varepsilon_i(\beta)\in\{\pm1\},
\]
\[
x_\beta(r)x_\delta(s)
=x_\delta(s)x_\beta(r)x_{\beta+\delta}(N_{\beta,\delta}rs)
\quad(\beta+\delta\in\Phi),
\]
with commutation when $\beta+\delta\notin\Phi$, and the rank-one identity
\[
x_\beta(r)x_{-\beta}(s)
=x_{-\beta}\!\left(\frac{s}{1+rs}\right)
\beta^\vee(1+rs)
x_\beta\!\left(\frac{r}{1+rs}\right).
\]
Since $E_6$ is simply laced, $N_{\beta,\delta}=\pm1$. In the collection below
each denominator which occurs is a unit in $R_0$, so every step is an equality
in $G(R_0)$, not a numerical specialization.

Fix once and for all an order on $\Phi_M^+$, an order on $\Phi_N$ refining
height, and the opposite order on $-\Phi_N$. With these choices, the
multiplication map
\[
U_M\times A\times U_M\times N\times\overline{N}\longrightarrow G,
\qquad
(u_1,t,u_2,n',\overline{n})\longmapsto
u_1t\dot{w}_Mu_2n'\overline{n}
\]
is an open immersion on the cell under consideration. 
Hence such a decomposition is unique whenever it exists. The table below 
records the corresponding decomposition of the product
\[
\dot{w}_0^{-1}x_{\gamma_0}(c_0)x_{\gamma_1}(c_1)x_{\gamma_2}(c_2)
x_{\gamma_3}(c_3)x_{\gamma_4}(c_4)x_{\gamma_5}(c_5)
\]
into the above ordered cell, here we view the coordinates $c_i$ as indeterminate. All other coordinates in
the two $U_M$-components which are not presented are zero, and the remaining collected factors have roots
only in $\Phi_N$ and $\Phi_N^-$; hence they define elements
$n'\in N(R_0)$ and $\overline{n}\in\overline{N}(R_0)$.
\[
\begin{array}{c|c|c}
\text{Component} & \text{coordinate} & \text{value} \\ \hline
u_1 & x_{\alpha_{23}} & Xc_4^{-1}c_5^{-1}\\
u_1 & x_{\alpha_{15}} & Xc_3^{-1}c_5^{-1}\\
u_1 & x_{\alpha_4} & -Xc_2^{-1}c_5^{-1}\\
t & \alpha_1^\vee & -c_4c_5\\
t & \alpha_2^\vee & Y\\
t & \alpha_3^\vee & -c_3c_4c_5^2\\
t & \alpha_4^\vee & -c_2c_3c_4c_5^3\\
t & \alpha_5^\vee & -c_3c_4c_5^2\\
t & \alpha_6^\vee & -c_4c_5\\
u_2 & x_{\alpha_{23}} & c_0c_4^{-1}c_5^{-1}\\
u_2 & x_{\alpha_{15}} & c_0c_3^{-1}c_5^{-1}\\
u_2 & x_{\alpha_4} & -c_0c_2^{-1}c_5^{-1}
\end{array}
\]
Let us spell out how this finite table proves the asserted decomposition. The
collection algorithm proceeds by moving reduced expressions for $\dot{w}_0^{-1}$ from \eqref{eqn: simple-reflection-w0} across the six root factors, applying the rank-one
identity exactly when a root is paired with its negative, and otherwise using
the Chevalley commutator relation to restore the fixed order. At the end of
this process the ordered word has the form
\[
u_1t\dot{w}_Mu_2n'\overline{n}.
\]
 Since the above multiplication map is an
open immersion, there is no second possible set of coordinates. This gives the
identity
\[
\dot{w}_0^{-1}n=u_1t\dot{w}_Mu_2n'\overline{n}
\]
in $G(R_0)$, and therefore after specialization to $F$. The displayed values
are exactly the formulas for $u_1$, $t$, and $u_2$ in the statement. Since
$c_i\ne0$, $X\ne0$, and $Y\ne0$, every simple coroot parameter in $t$ is
nonzero. Hence $m=u_1t\dot{w}_Mu_2$ lies in the big Bruhat cell of $M$.

The same collected word records the first coordinate of the opposite
unipotent factor. If
\[
\dot{w}_0^{-1}\overline{n}\dot{w}_0
=\prod_{\beta\in\Phi_N}^{\rightarrow}x_\beta(y_\beta)
\]
in a height-refining order, then the collection gives the additional
coordinate
\[
\begin{array}{c|c|c}
\text{factor} & \text{coordinate} & \text{value} \\ \hline
\dot{w}_0^{-1}\overline{n}\dot{w}_0
& x_{\alpha_2} & -c_5XY^{-1}
\end{array}
\]
and hence
\[
y_{\alpha_2}=-\frac{c_5X}{Y}
=\frac{-c_0c_5+c_1c_5^2}
{c_0^2-c_0c_1c_5-c_2c_3c_4c_5}.
\]
This coordinate is independent of the chosen height-refining order:
$\alpha_2=(010000)$ is the minimal root in $\Phi_N$, and commutators of
positive root subgroups in $N$ can only contribute to root subgroups of
strictly larger height. Thus no later collection step can change the
$x_{\alpha_2}$-coordinate. This is the coefficient of $x_\alpha$ in the
statement, since $\alpha=\alpha_2$ for the parabolic under consideration.
\end{proof}

\begin{cor}{\label{finiteetale1}}The induced map of $\mathcal{N}$ on $U_M\backslash N'$ is \'etale onto its image of degree 2. 
\end{cor}
\begin{proof} Let $t_1=-c_4c_5, t_2=c_0(c_0-c_1c_5)-c_2c_3c_4c_5, t_3=-c_3c_4c_5^2, t_4=-c_2c_3c_4c_5^3, t_5=(c_0-c_1c_5)c_4^{-1}c_5^{-1}, t_6=c_0c_4^{-1}c_5^{-1}, t_7=(c_0-c_1c_5)c_3^{-1}c_5^{-1}, t_8=-(c_0-c_1c_5)c_2^{-1}c_5^{-1}, t_9=c_0c_3^{-1}c_5^{-1}, t_{10}=-c_0c_2^{-1}c_5^{-1}$. Then $c_0-c_1c_5=-t_1t_5, c_3c_5=\frac{t_3}{t_1}, c_2c_5=\frac{t_4}{t_3},c_0=-t_1t_6$, thus
$$t_7=-\frac{t_1^2t_5}{t_3}, t_8=\frac{t_1t_3t_5}{t_4}, t_9=-\frac{t_1^2t_6}{t_3}, t_{10}=\frac{t_1t_3t_6}{t_4}.$$ It follows that the induced map $\mathcal{N}$
on $U_M\backslash N'$ is completely determined by 
$$(c_0,c_1,c_2,c_3,c_4,c_5)\mapsto (t_1,t_2,t_3,t_4,t_5,t_6).$$ The map clearly factors through the action of $\mathbb{Z}/2\mathbb{Z}$ by
$(c_0,c_1,c_2,c_3,c_4,c_5)\mapsto (c_0,-c_1,-c_2,-c_3,-c_4,-c_5)$, and the determinant of its Jacobian matrix is $$\frac{c_3c_5^3}{c_4}(2c_0^2(c_3c_4+c_2c_5)-2c_2c_3^2c_4^2c_5+c_0c_1c_5^2(c_4-c_2))$$
Given $c_i\neq 0$ for $i=0,\cdots, 5$ and shrink $N'$ if necessary by assuming that the polynomial in the bracket is non-zero. By further assuming that $t_2-t_1^2t_5t_6\neq 0$, the fiber of $\mathcal{N}$ is uniquely determined by the algebraic relations:
\[c_1=-\frac{t_1(t_5-t_6)}{c_5},c_2=\frac{t_4}{t_3c_5}, c_3=\frac{t_3}{t_1c_5}, c_4=-\frac{t_1}{c_5}, c_5^2=\frac{t_4}{t_2-t_1^2t_5t_6}\]
From this, it is immediate that the induced map $\mathcal{N}$ on $U_M\backslash N'$ is \'etale of degree 2 onto its image.
\end{proof}

\begin{remark}We expect that the induced map $\mathcal{N}$ on $U_M\backslash N'$ is finite \'etale onto its image for general self-associate parabolic subgroup $P=MN$ of $G$ such that the action of $U_M$ on $N'$ have trivial stabilizer for generic $n\in N'$. However we have not yet found a uniform proof.  
    
\end{remark}

From now on we set $R$ to be the Zariski dense subset of the affine space $\mathbb{A}^6$ defined by \[R=\{(c_0,\cdots, c_5): c_i\neq 0, c_0-c_1c_5\neq 0, c_0(c_0-c_1c_5)-c_2c_3c_4c_5\neq 0, 2c_0^2(c_3c_4+c_2c_5)-2c_2c_3^2c_4^2c_5+c_0c_1c_5^2(c_4-c_2)\neq 0\}.\] After replacing $N'$ by the corresponding saturated open subset, Theorem \ref{orbit-space-representative}, Proposition \ref{BruhatDecomp}, and Corollary \ref{finiteetale1} identify $R$ as a set of orbit space representatives for the geometric quotient $U_M\backslash N'$.

\begin{cor}{\label{finiteetale2}} The map $\mathcal{N}: N'\rightarrow M$ is \'etale onto its image of degree 2.
\end{cor}
\begin{proof}
     By Theorem \ref{orbit-space-representative}, the action of $U_M$ has trivial stabilizer. Therefore we obtain a commutative diagram
    \[ \begin{tikzcd}
U_M\times R \arrow{r} \arrow[swap]{d}{\mathrm{Id}_{U_M}\times \mathcal{N}} & N' \arrow{d}{\mathcal{N}} \\%
U_M\times \mathcal{N}(R) \arrow{r}& \mathcal{N}(N')
\end{tikzcd}
\] where the top horizontal map is the action of $U_M$ on $N$ by conjugation, and the bottom horizontal map is the $\mathrm{Ad}(\dot{w}_0^{-1})$-twisted conjugate action of $U_M$ on $M$. Both of the horizontal maps are isomorphisms.  By Corollary \ref{finiteetale1}, the left vertical map is \'etale of degree 2, so is the right vertical map.  
\end{proof}

\subsection{Local coefficient formula}
In this section, we will obtain an explicit formula for the local coefficient in our case in terms of the Mellin transform of certain partial Bessel functions. 
\subsubsection{Preparations}{\label{PrepLCF}}
Under Assumption 4.1 and 5.1 of \cite{Sha02}, it is proved by Theorem 6.2 of \cite{Sha02} that the local coefficient $C_\psi(s,\sigma)$ can be written as an integral of certain partial Bessel functions over the p-adic manifold obtained by the $F$-points of $Z_M^0U_M\backslash N'$. Let us first verify the assumptions in our case. From now on we denote $\alpha=\alpha_2$, and let $Z_G$ be the center of $G$.
\begin{lem}{\label{Assumption5-1}} The map $\alpha^\vee : F^\times\overset{\gamma_0^\vee}{\rightarrow} Z_M\rightarrow Z_G\backslash Z_M$ is injective and $\alpha(\alpha^\vee(t))=t$ for all $t\in F^\times$.    
\end{lem}
\begin{proof} If $\gamma_0^\vee(t)\in Z_G=\cap_{\beta\in \Delta}\ker \beta$, then $t^{\langle\alpha_2,\gamma_0^\vee\rangle}=t=1$. This shows that $\alpha^\vee$ is injective. The second statement of the lemma follows easily from $\alpha(\alpha^\vee(t))=t^{\langle\alpha_2,\gamma_0^\vee\rangle}=t$.     
\end{proof} 

This verifies the assumption 5.1 of \cite{Sha02}. Denote the image of $\alpha^\vee$ by $Z_M^0$. The local coefficient is an integral over the orbit space $Z_M^0U_M\backslash N'$, where $Z_M^0$ acts on $U_M\backslash N'$ by conjugation. We have
$$\alpha^\vee(t)(\prod_{i=0}^5x_{\gamma_i}(c_i))\alpha^\vee(t)^{-1}=\prod_{i=0}^5x_{\gamma_i}(t^{\langle \gamma_i, \gamma_0^\vee\rangle}c_i)=x_{\gamma_0}(t^2 c_0)\prod_{i=1}^5x_{\gamma_i}(t c_i)$$
The last equality holds since $\langle \alpha_i,\gamma_0^\vee\rangle=0$ for $i=1,3,4,5,6$, and $\langle \alpha_2,\gamma_0^\vee\rangle=1$, which implies that $\langle\gamma_i,\gamma_0^\vee\rangle=1$ for $i=1,3,4,5,6$, and $\langle \gamma_0,\gamma_0^\vee\rangle=2$. From this we also obtain that
\begin{lem}\label{etale-fundamental-group} The \'etale fundamental group of the map $\mathcal{N}$ on $U_M\backslash N'$ onto its image in Corollary \ref{finiteetale1}, hence the one for $\mathcal{N}: N'\rightarrow M$ in Corollary \ref{finiteetale2} is $\mathbb{Z}/2\mathbb{Z}$ and generated by $\alpha^\vee(-1)$.
    
\end{lem}

\begin{proof}
   The above equality immediately implies that the adjoint action of $\alpha^\vee(-1)$ takes $(c_0,c_1,c_2,c_3,c_4,c_5)$ to $(c_0,-c_1,-c_2,-c_3,-c_4,-c_5)$, as in the proof of Corollary \ref{finiteetale1}.
\end{proof}

We may choose $t=c_5^{-1}$, a set of orbit space representatives of $Z_M^0 U_M\backslash N'$ can therefore be given by
$$\gamma_0=(122321), \gamma_1=(112321),\gamma_2=(112221),\gamma_3=(111211),\gamma_4=(011210)$$
Identify the $F$-points of the orbit space 
$Z_M^0U_M\backslash N'$ by the open dense subset of 5-dimensional affine space obtained by setting $c_5=1$ in $R$, and denote it by $R'$. To be precise,
\[R'=\{(c_0,\cdots, c_4): c_i\neq 0, c_0-c_1\neq 0, c_0(c_0-c_1)-c_2c_3c_4\neq 0, 2c_0^2(c_3c_4+c_2)-2c_2c_3^2c_4^2+c_0c_1(c_4-c_2)\neq 0\}.\]
Then the corresponding 
invariant measure is
$$d\dot{n}=\vert c_3\vert^2\vert c_4\vert^4\prod_{i=0}^4dc_i$$

\begin{cor}\label{finiteetale3}The induced map $\mathcal{N}: \dot{n}\mapsto \dot{m}$ via the Bruhat decomposition $\dot{w}_0^{-1}\dot{n}=\dot{m}\dot{n}'\overline{\dot{n}}$ on $Z_M^0U_M\backslash N'$ is an isomorphism onto its image in $M$, up to a Zariski open dense subset. 
\end{cor}
\begin{proof} It is elementary to see that the map $\mathcal{N}$ on $Z_M^0U_M\backslash N'$
\[\underline{c}:=(c_0,c_1,c_2,c_3, c_4)\mapsto \underline{t}:=(t'_1,t'_2, t'_3, t'_5, t'_6)\in 
\mathbb{A}^5,\]where the  
$t'_i$'s are obtained by setting $c_5=1$ in the corresponding expression of $t_i$ in Corollary \ref{finiteetale1}, 
completely determines the image of $\mathcal{N}$. Again by the computation in Corollary \ref{finiteetale1}, it suffices to show that $t_4'$ can be determined by $t'_1,t'_2,t'_3,t'_5,t'_6$. In fact, one checks that $t_4'=t_2'-t_1'^2t_5't_6'$. By further assuming that $t_2'-t_1'^2t_5't_6'\neq 0$, the above map admits an inverse, given by
\[c_0=-t_1't_6', c_1=t_1'(t_5'-t_6'), c_2=\frac{t_2'-t_1'^2t_5't_6'}{t_3'}, c_3=\frac{t_3'}{t_1'}, c_4=-t_1'\] One also computes that the determinant of the Jacobian of the above map is $-c_3=-\frac{t_3'}{t_1'}$.
\end{proof}

\subsubsection{Partial Bessel functions}\label{partial-Bessel-function}
We keep the notation of the preceding subsections. In particular,
$P=MN$ is the maximal parabolic subgroup of the split simply connected
$E_6$-group obtained by deleting $\alpha_2$, the group $N$ is the
unipotent radical of $P$, and $\gamma_0=(122321)$ is the highest root.
Recall that $N$ is a Heisenberg group and that its center is the root
subgroup $U_{\gamma_0}$. We write
\[
\Phi_N^{\circ}=\Phi_N\setminus\{\gamma_0\}.
\]
Thus the roots of $\overline N$ are
$\Phi_N^-=-\Phi_N$, and we put
\[
\Phi_N^{-,\circ}=\Phi_N^-\setminus\{-\gamma_0\}.
\]
Let $v_F:F^\times\to \bZ$ be the additive valuation normalized by
$v_F(\varpi)=1$, extended by $v_F(0)=+\infty$, and put
\[
\mathfrak p^r=\{a\in F: v_F(a)\ge r\},\qquad r\in\bZ.
\]
For a root $\beta$, set
\[
U_{\beta,r}=x_\beta(\mathfrak p^r).
\]
We shall use the same element $x=2\rho^\vee\in X_*(A)$ as in the
Bruhat--Tits notation \cite{BT72}, but only to specify the depths of the root subgroups.
No exponential map is used below.

For $\kappa\in\bZ$, define the opposite root-depth function on $\Phi_N^-$ by
\[
\overline r_{x,\kappa}(\eta)=
\begin{cases}
-\kappa-\eta(x), & \eta\in \Phi_N^{-,\circ},\\[4pt]
-2\kappa-\eta(x), & \eta=-\gamma_0.
\end{cases}
\]
This is the root-subgroup analogue of the Moy--Prasad notation
$\overline{\mathfrak n}_{x,\kappa}$, with the central root depth doubled.
Equivalently, at the level of root lattices one may write
\[
\overline{\mathfrak n}_{x,\kappa}
=
\bigoplus_{\eta\in\Phi_N^-}\mathfrak g_{\eta,\overline r_{x,\kappa}(\eta)},
\]
where $\mathfrak g_{\eta,r}$ denotes the usual depth-$r$ root lattice.
Choose once and for all a total ordering of $\Phi_N^-$ in which
$-\gamma_0$ is placed last. Define
\[
\overline{N}_{0,\kappa}
=
\prod_{\eta\in \Phi_N^-}^{\longrightarrow}
U_{\eta,\overline r_{x,\kappa}(\eta)}
\subset \overline{N}(F),
\]
where the arrow indicates the chosen order. For the noncentral roots this
is exactly the Moy--Prasad condition $\eta(x)+r\ge -\kappa$; the central
coordinate is given twice the $\kappa$-depth. This doubling is the point
which is needed for the subgroup property.

Similarly, define
\[
U_0
=
U_M\cap G_{x,0}
=
\prod_{\alpha\in\Phi_M^+}^{\longrightarrow} U_{\alpha,-\alpha(x)}
\subset U_M(F).
\]
This is the root-subgroup description of the Moy--Prasad subgroup
$U_M\cap G_{x,0}$ \cite{PM94}.

\begin{lem}\label{lem:root-subgroup-filtration-opposite-N}
For every $\kappa\in\bZ$, the set $\overline{N}_{0,\kappa}$ is an open
compact subgroup of $\overline{N}(F)$. For $\kappa_1\le \kappa_2$, one has
$\overline{N}_{0,\kappa_1}\subset \overline{N}_{0,\kappa_2}$, and
\[
\bigcup_{\kappa\ge 0} \overline{N}_{0,\kappa}=\overline{N}(F).
\]
Moreover $\overline{N}_{0,\kappa}$ is normalized by $U_0$. Finally, if
$\alpha^\vee:F^\times\to Z_M$ is the cocharacter of Lemma
\ref{Assumption5-1}, then
\[
\alpha^\vee(t)\overline{N}_{0,\kappa}\alpha^\vee(t)^{-1}
=\overline{N}_{0,\kappa+v_F(t)}.
\]
In particular $\alpha^\vee(t)\overline{N}_{0,\kappa}\alpha^\vee(t)^{-1}$
depends only on $v_F(t)$, equivalently only on $|t|$.
\end{lem}

\begin{proof}
Since the product map
\[
\prod_{\eta\in\Phi_N^-}^{\longrightarrow} U_{\eta}(F)
\longrightarrow \overline{N}(F)
\]
is a homeomorphism for any fixed ordering of the root subgroups, the set
$\overline{N}_{0,\kappa}$ is open and compact once it is known to be a
subgroup.

We prove the subgroup assertion by the Chevalley commutator relations.  With respect to the fixed
pinning, all Chevalley commutator constants are integers in $F$, hence have
nonnegative valuation. The group $\overline{N}$ is two-step nilpotent and
its commutator subgroup is contained in $U_{-\gamma_0}$. Thus the only
possible nontrivial commutators between root groups in $\overline{N}$ are
of the form
\[
[x_\eta(a),x_{\eta'}(b)]
=x_{-\gamma_0}(C_{\eta,\eta'}ab),
\qquad \eta,\eta'\in\Phi_N^{-,\circ},\quad
\eta+\eta'=-\gamma_0,
\]
with $C_{\eta,\eta'}\in\bZ$. If
$a\in\mathfrak p^{\overline r_{x,\kappa}(\eta)}$ and
$b\in\mathfrak p^{\overline r_{x,\kappa}(\eta')}$, then
\[
v_F(C_{\eta,\eta'}ab)
\ge \overline r_{x,\kappa}(\eta)+\overline r_{x,\kappa}(\eta')
=-2\kappa-(\eta+\eta')(x)
=-2\kappa+\gamma_0(x)
=\overline r_{x,\kappa}(-\gamma_0).
\]
Therefore all commutators produced in the collection process remain inside
$U_{-\gamma_0,\overline r_{x,\kappa}(-\gamma_0)}$, and
$\overline{N}_{0,\kappa}$ is a subgroup.

The monotonicity in $\kappa$ follows from
\[
\overline r_{x,\kappa_1}(\eta)\ge \overline r_{x,\kappa_2}(\eta)
\]
whenever $\kappa_1\le \kappa_2$, and
the union statement is immediate from the root subgroup coordinates of
$\overline{N}(F)$.

We next prove normalization by $U_0$. Let
$x_\alpha(a)\in U_{\alpha,-\alpha(x)}$, with $\alpha\in\Phi_M^+$, and let
$x_\eta(b)\in U_{\eta,\overline r_{x,\kappa}(\eta)}$, with
$\eta\in\Phi_N^-$. If $\eta=-\gamma_0$, then no root of the form
$-\gamma_0+i\alpha$, $i>0$, occurs: $\gamma_0$ is the unique root in
$\Phi_N$ with $\alpha_2$-coefficient $2$. Hence the central opposite root
subgroup is normalized. If $\eta\in\Phi_N^{-,\circ}$, then every root of the
form $\eta+i\alpha$ which occurs in the Chevalley commutator relation again
belongs to $\Phi_N^{-,\circ}$, because $\alpha$ has $\alpha_2$-coefficient
zero. For such a root we have
\[
v_F(a^i b)
\ge -i\alpha(x)-\kappa-\eta(x)
=-\kappa-(\eta+i\alpha)(x)
=\overline r_{x,\kappa}(\eta+i\alpha).
\]
Thus conjugation by each generator of $U_0$ preserves
$\overline{N}_{0,\kappa}$, and so $U_0$ normalizes
$\overline{N}_{0,\kappa}$.

It remains to check the behavior under $\alpha^\vee(t)$. The cocharacter
$\alpha^\vee$ acts with weight $-1$ on all noncentral opposite root groups
and with weight $-2$ on the central opposite root group $U_{-\gamma_0}$.
Hence
\[
\alpha^\vee(t)x_\eta(a)\alpha^\vee(t)^{-1}
=
\begin{cases}
x_\eta(t^{-1}a), & \eta\in\Phi_N^{-,\circ},\\[3pt]
x_{-\gamma_0}(t^{-2}a), & \eta=-\gamma_0.
\end{cases}
\]
Consequently
\[
U_{\eta,\overline r_{x,\kappa}(\eta)}
\longmapsto U_{\eta,\overline r_{x,\kappa}(\eta)-v_F(t)}
=U_{\eta,\overline r_{x,\kappa+v_F(t)}(\eta)},
\qquad \eta\in\Phi_N^{-,\circ},
\]
and
\[
U_{-\gamma_0,\overline r_{x,\kappa}(-\gamma_0)}
\longmapsto
U_{-\gamma_0,\overline r_{x,\kappa}(-\gamma_0)-2v_F(t)}
=U_{-\gamma_0,\overline r_{x,\kappa+v_F(t)}(-\gamma_0)}.
\]
This proves the asserted equality.
\end{proof}

Let $\pi$ be supercuspidal and put $\sigma=\xi^*(1\otimes\pi)$. Choose
$f\in C_c^\infty(M;\omega_\sigma)$, compactly supported modulo $Z_M$,
with
\[
f(zm)=\omega_\sigma(z)f(m),\qquad z\in Z_M,\ m\in M,
\]
such that
\[
W_f(m)=\int_{U_M} f(um)\psi^{-1}(u)\,du
\]
is a nonzero Whittaker function for $\sigma$. We normalize $f$ so that
$W_f(e)=1$.
For $m\in M$, put
\[
U'_{M,m}
=
\{u\in U_M: mum^{-1}\in U_M,\ \psi(mum^{-1})=\psi(u)\}.
\]
The associated Bessel function is
\[
j_\sigma(m)
=
\int_{U'_{M,m}\backslash U_M} W_f(mu)\psi^{-1}(u)\,du.
\]
It satisfies
\[
j_\sigma(u_1mu_2)=\psi(u_1u_2)j_\sigma(m),
\qquad u_1,u_2\in U_M.
\]

Let $n\in N'$ be such that the Bruhat decomposition
\[
\dot{w}_0^{-1}n=mn'\overline{n}
\]
holds, with $m\in M$, $n'\in N$, and
$\overline{n}\in\overline{N}$. Let
\[
U_{M,n}=\{u\in U_M: unu^{-1}=n\}.
\]
Then $U_{M,n}\subset U'_{M,m}$. We use the following structural input of
Sundaravaradhan \cite{Sun07}.

\begin{lem}\label{Sundaravaradhan}
For $n\in N'$ satisfying $\dot{w}_0^{-1}n=mn'\overline{n}$, one has
\[
U_{M,n}=U'_{M,m}
\]
away from a subset of measure zero in $N'$.
\end{lem}

For $m,z$ arising from such an $n$, and for $\kappa>0$, define the partial
Bessel function with cutoff $\overline{N}_{0,\kappa}$ by
\[
j_{\sigma,\kappa}(m,z)
=
\int_{U_{M,n}\backslash U_M}
W_f(mu)\,
\mathbf 1_{\overline{N}_{0,\kappa}}(zu^{-1}\overline{n}uz^{-1})\,
\psi^{-1}(u)\,du.
\]
Here $\mathbf 1_{\overline{N}_{0,\kappa}}$ denotes the characteristic
function of $\overline{N}_{0,\kappa}$.

Now let
\[\dot{n}=
\dot{n}(\underline{c})=\prod_{i=0}^4 x_{\gamma_i}(c_i),
\qquad \underline{c}=(c_0,\ldots,c_4)\in R',
\]
be the chosen representative of $Z_M^0U_M\backslash N'$, and write
\[
\dot{w}_0^{-1}\dot{n}(\underline{c})
=\dot{m}(\underline{c})\dot{n}'(\underline{c})\overline{\dot{n}}(\underline{c}),
\qquad
\dot{m}(\underline{c})=\mathcal N(\dot{n}(\underline{c})).
\]
Let $\dot{x}_\alpha(\underline{c})\in F^\times$ be the $x_\alpha$-coordinate
of $\dot{w}_0^{-1}\overline{n}(\underline{c})\dot{w}_0$. Equivalently, by
Proposition \ref{BruhatDecomp} after setting $c_5=1$,
\[\dot{x}_\alpha=
\dot{x}_\alpha(\underline{c})
=
\frac{c_1-c_0}{c_0^2-c_0c_1-c_2c_3c_4}.
\]
Let $d$ and $g$ be the conductor of $\psi$ and $\omega_\sigma^{-1}(\dot{w}_0\omega_\sigma)$ respectively, and put
\[
z(\underline{c})
=\alpha^\vee\bigl(\varpi^{d+g}\dot{x}_\alpha(\underline{c})\bigr)
\in Z_M^0.
\]
For $\kappa$ sufficiently large, the partial Bessel function used below is
\[
j_{\sigma,\kappa}(\underline{c})
=
j_{\sigma,\kappa}(\dot{n}(\underline{c}))
:=
j_{\sigma,\kappa}\bigl(\dot{m}(\underline{c}),z(\underline{c})\bigr)
=
\int_{U_{M,\dot{n}(\underline{c})}\backslash U_M}
W_f(\dot{m}(\underline{c})u)\,
\mathbf 1_{\overline{N}_{0,\kappa}}
\bigl(z(\underline{c})u^{-1}\overline{\dot{n}}(\underline{c})u
z(\underline{c})^{-1}\bigr)
\psi^{-1}(u)\,du,
\qquad \underline{c}\in R'.
\] From now on, we will omit $\underline{c}$ to simplify the notations.

\subsubsection{Local coefficients as Mellin transform of partial Bessel functions}
By Lemma \ref{Assumption5-1}, Assumption 5.1 of \cite{Sha02} is satisfied. Thanks to the work of Sundaravaradhan \cite{Sun07}, Assumption 4.1 of \cite{Sha02} is also valid. For the convenience of the reader, let us provide the statement of Theorem 6.2 of \cite{Sha02}:
\begin{thm}Let $\sigma$ be a $\psi$-generic supercuspidal representations of $M$. Choose $f\in C^\infty_c(M;\omega_\sigma)$ such that $W_f$ is a non-zero Whittaker function of $\sigma$, normalized so that $W_f(e)=1$. Let $\widetilde{\alpha}:=\langle\rho, \alpha\rangle^{-1}\alpha$, where $\rho$ is half sum of the positive roots in $N$. $\alpha^\vee$ is the map in Lemma \ref{Assumption5-1}. Then
\[
C_\psi(s,\sigma)^{-1}=\gamma(2\langle\tilde{\alpha},\alpha^\vee\rangle s, \omega_\sigma(\dot{w}_0\omega_\sigma^{-1}), 
\psi^{-1})\int_{Z_M^0U_M\backslash N'}j_{\sigma,\kappa}(\dot{n})\omega_{\sigma_s}^{-1}(\dot{w}_0\omega_{\sigma_s})(\dot{x}_\alpha)q^{\langle s\tilde{\alpha}+\rho, H_M(\dot{m})\rangle}d\dot{n}
\]
for sufficiently large $\kappa$,
where $\dot{m}$ is the image of $\dot{n}$ via $\dot{w}_0^{-1}\dot{n}=\dot{m}\dot{n}'\bar{\dot{n}}$, which holds off a subset of measure zero on $N$. Here $d\dot{n}$ is the invariant measure on the orbit space $Z_M^0U_M\backslash N'$, $\sigma_{s}=\sigma\otimes q^{\langle s\tilde{\alpha}, H_M(\cdot)\rangle}$,  $\dot{x}_\alpha=u_{\alpha}(\dot{w}_0\bar{\dot{n}}\dot{w}_0^{-1})$, and $\gamma(2\langle\tilde{\alpha},\alpha^\vee\rangle s, \omega_\sigma(\dot{w}_0\omega_\sigma^{-1}), \psi^{-1})$ is the Abelian $\gamma$-factor depending only on $\omega_\sigma$. 
    
\end{thm}

Let us adapt the theorem to our case.
We have $$\widetilde{\alpha}=\langle \rho, \alpha\rangle^{-1}\rho=\langle \frac{11}{2}\gamma_0, \alpha_2\rangle^{-1}(\frac{11\gamma_0}{2})=\gamma_0$$
So $\langle\widetilde{\alpha},\alpha^\vee\rangle=\langle \gamma_0,\gamma_0^\vee\rangle=2$. Moreover,
$$q^{\langle s\widetilde{\alpha}+\rho, H_M(\cdot)\rangle}=q^{\langle (s+\frac{11}{2})\gamma_0, H_M(\cdot)\rangle}=\vert \gamma_0(\cdot)\vert^{s+\frac{11}{2}}$$ By Lemma \ref{w_0-action-on-roots}, $\dot{w}_0(\alpha_i)=\alpha_i$ for all $i\neq 2$, and   $\dot{w}_0(\alpha_2)=-\gamma_1$, we obtain   $\dot{w}_0(\alpha^\vee)=\dot{w}_0(\gamma_0^\vee)=-\gamma_0^\vee$. Consequently, we also have
$$\omega_\sigma(\dot{w}_0\omega_\sigma^{-1})(\alpha^\vee(t))=\omega_\sigma(\gamma_0^\vee(t))\omega_\sigma^{-1}(\dot{w}_0\gamma_0^\vee(t)\dot{w}_0^{-1})=\omega_\pi^2(t)$$
$$\omega_{\sigma_s}^{-1}(\dot{w}_0\omega_{\sigma_s})(\alpha^\vee(t))=\omega_{\sigma_s}^{-2}(\alpha^\vee(t))=\omega_\sigma^{-2}(\alpha^\vee(t))\vert\gamma_0(\alpha^\vee(t))\vert^{-2s}=\omega_\pi^{-2}(t) \vert t\vert^{-4s}$$ where
$\sigma_s=\sigma\otimes q^{\langle s\widetilde{\alpha}, H_M(\cdot)\rangle}$.

Let $\dot{n}=\prod_{i=0}^4x_
{\gamma_i}(c_i)$ be a representative of $Z_M^0U_M
\backslash N'$ in $N'$, and decompose $\dot{w}_0^{-1}\dot{n}$ as
$$\dot{w}_0^{-1}\dot{n}=\dot{m}\dot{n}'\overline{\dot{n}}$$ There exists a unique element $\dot{x}_\alpha\in F$, independent of the expression of $\dot{n}$, such that$$\psi(\dot{w}_0^{-1}\overline{\dot{n}}\dot{w}_0)=\psi(\dot{x}_\alpha)$$ Let us be more precise to see that $\dot{x}_\alpha$ is well defined. Suppose $\dot{w}_0^{-1}\overline{\dot{n}}\dot{w}_0=\prod_{\beta\in \Phi_N^+}x_\beta(d_\beta)$, with $d_\beta\in F$. Then by definition, $\dot{x}_\alpha=d_\alpha$. From the relationship 
$$u_\gamma(x)u_\delta(y)u_{\gamma}(x)^{-1}=u_\delta(y)\prod_{i\gamma+j\delta\in \Phi_G, ,j>0}u_{i\gamma+j\delta}(c_{\gamma,\delta,i,j}x^iy^j)$$ (\cite{Springer1981LinearAG}), where $c_{\gamma,\delta, i,j}$ are structure constants uniquely determined by the roots $\gamma,\delta$, and the positive integers $i,j$, and since $\alpha$ is the minimal root with respect to the Bruhat order in $\Phi_N^+$, we see that the coefficient $d_\alpha$ for $\alpha$ does not change if we reorder the product $\prod_{\beta\in \Phi_N^+}x_\beta(d_\beta)$. From the above argument, we obtain
\begin{prop}
     Let $\sigma=\xi^*(\pi\otimes 1)$ where $\pi$ is a supercuspidal representations of $\mathrm{GL}_6(F)$, then
\begin{align*}
C_\psi(s,\sigma)^{-1}&=\gamma(4 s, \omega^2_\pi, 
\psi^{-1})\int_{Z_M^0U_M\backslash N'}j_{\sigma,\kappa}(\dot{n})\omega_\pi^{-2}(\dot{x}_\alpha)\vert \dot{x}_\alpha\vert^{-4s}\vert \gamma_0(\dot{m})\vert^{s+\frac{11}{2}}d\dot{n}
\end{align*}
for sufficiently large $\kappa$, where $\gamma(4 s, \omega^2_\pi, 
\psi^{-1})$ is the Tate $\gamma$-factor associated to $\omega_\pi^2$.
\end{prop}
 Since $\langle \gamma_0,\alpha_i\rangle=0$ for $i=1,3,4,5,6$, $\langle\gamma_0,\alpha_2^\vee\rangle=1$, and $\gamma_0(\dot{w}_M)=1$ by Lemma \ref{gamma_0(w_M)}, we get
 \[
\gamma_0(\dot{m})=\gamma_0(\dot{t})=\gamma_0(\alpha_1^\vee(-c_4)\alpha_2^\vee(c_0(c_0-c_1)-c_2c_3c_4)\alpha_3^\vee(-c_3c_4)\alpha_4^\vee(-c_2c_3c_4)\alpha_5^\vee(-c_3c_4)\alpha_6^\vee(-c_4))
 =c_0^2-c_0c_1-c_2c_3c_4.
 \] 
 Together with the invariant measure $d\dot{n}$ we derived from the previous section and the expression of $\dot{x}_\alpha$, we obtain
\begin{prop}\label{local-coeff-formula-1}
     Let $\sigma=\xi^*(\pi\otimes 1)$ where $\pi$ is a supercuspidal representations of $\mathrm{GL}_6(F)$, then
$$
C_\psi(s,\sigma)^{-1}=\gamma(4 s, \omega^2_\pi, 
\psi^{-1})\int_{(F^\times)^5}j_{\sigma,\kappa}(\underline{c})\omega_\pi^{-2}(\frac{c_1-c_0 }{c_0^2 - c_0c_1 - c_2c_3c_4})\vert \frac{c_1-c_0 }{c_0^2 - c_0c_1 - c_2c_3c_4}\vert^{-4s}$$$$\cdot\vert c_0^2-c_0c_1-c_2c_3c_4\vert^{s+\frac{11}{2}}\vert c_3\vert^2\vert c_4\vert^4\prod_{i=0}^4dc_i
$$
$$=\gamma(4 s, \omega^2_\pi, 
\psi^{-1})\int_{(F^\times)^5}j_{\sigma,\kappa}(\underline{c})\omega_\pi^{-2}(\frac{c_1-c_0 }{c_0^2 - c_0c_1 - c_2c_3c_4})\vert c_1-c_0 \vert^{-4s}$$$$\cdot\vert c_0^2-c_0c_1-c_2c_3c_4\vert^{5s+\frac{11}{2}}\vert c_3\vert^2\vert c_4\vert^4\prod_{i=0}^4dc_i
$$
for sufficiently large $\kappa$, where $\gamma(4 s, \omega^2_\pi, 
\psi^{-1})$ is the Tate $\gamma$-factor associated to $\omega_\pi^2$.
\end{prop}

\subsection{Asymptotics of partial Bessel functions}
To prove analytic stability, a crucial step is to obtain an asymptotic expansion for the partial Bessel functions appearing in the local coefficient formula. As we see in the previous section, these functions are defined using the Bruhat decomposition $\dot{w}_0^{-1}n = mn'\overline{n}$. By Corollary \ref{finiteetale1}, the map $n \mapsto m$ is finite \'etale onto its image when passing to the quotient $U_M\backslash N'$, which allows us to connect the partial Bessel functions to partial Bessel integrals. The latter are functions on $M$ admitting asymptotic expansions in terms of Bruhat cells of $M$, and do not depend on the decomposition $\dot{w}_0^{-1}n=mn'\overline{n}$.

This shift in perspective is advantageous: the partial Bessel integrals, when restricted to the Bruhat cells of $M$, possess well-behaved asymptotic expansions. It is this key property that ultimately enables the proof of analytic stability.

\subsubsection{Partial Bessel integrals}\label{partial-Bessel-integral}Throughout this subsection, we adopt the convention that $M$ denotes both a split connected reductive group over $F$ and its group of $F$-points. We will specify the intended meaning only if the context is ambiguous.

Given a supercuspidal representation $\sigma$ of $M$, choose $f\in C^\infty_c(M;\omega_\sigma)$ and consider the Whittaker model $W_f$, normalized so that $W_f(e)=1$ as before. Let $\Theta_M\in\mathrm{Aut}(M)$ be an $F$-involution on $M$, fixing the splitting $(B_M, A , \{x_\alpha\}_{\alpha\in\Delta_M})$, where we fix a Borel subgroup $B_M=AU_M$ defined over $F$ with $A$ the maximal (split) torus and $U_M$ the unipotent radical of $B_M$, and $\Delta_M$ is the set of simple roots. Assume that $\Theta_M$ is compatible with $\psi$, i.e., $\psi(\Theta_M(u))=\psi(u)$ for all $u\in U_M$. Define the partial Bessel integral as
\begin{align*}
B^M_\varphi(m,f):&=\int_{U_{M,m}^{\Theta_M}\backslash U_M}W_f(mu)\varphi(\Theta_M(u^{-1})m'u)\psi^{-1}(xu)du \\&=\int_{U_{M,m}^{\Theta_M}\backslash U_M}\int_{U_M}f(xmu)\varphi(\Theta_M(u^{-1})m'u)\psi^{-1}(xu)dxdu
\end{align*}
where $U_{M,m}^{\Theta_M}=\{m\in U_M: \Theta_M(u^{-1})mu=m\}$ is the twisted centralizer of $m$ in $U_M$, $\varphi$ is the characteristic function of certain subset of $M(F)$, assumed to be invariant under translation by elements of the center with absolute value 1.  $m'$ is a representative in $M$ of $m$ under the action of $ Z_M$ by $m\mapsto \Theta_M(z)mz^{-1}$, $z\in Z_M$. Note that since $\Theta_M$ preserves $\Delta_M$, it also preserves the center $Z_M$ of $M$, as $Z_M=\cap_{\alpha\in \Delta_M}\ker \alpha$. Suppose $m=\Theta_M(z_1)m'_1z_1^{-1}=\Theta_M(z_2)m'_2z_2^{-1}$, then $\Theta_M(z_1)z_1^{-1}\Theta_M(z_2^{-1})z_2=m_2'm_1'^{-1}\in Z_M\cap M_D$, thus $m_2'=m_1'\zeta$ for some $\zeta\in Z_M\cap M_D$. Since $Z_M\cap M_D$ is finite, $\zeta^n=1$ for some $n\in \mathbb{N}$, thus $\vert \zeta\vert =1$, and $\varphi(\Theta_M(u)^{-1}m'_2u)=\varphi(\zeta\Theta_M(u)^{-1}m'_1u)=\varphi(\Theta_M(u)^{-1}m_1'u)$ by our assumption on $\varphi$. Let us show that the integral is well-defined. Replace $u$ by $u'u$ with $u'\in U_{M,m}^{\Theta_M}$, since $\Theta_M$ preserves $Z_M$, $U_{M,m}^{\Theta_M}=U_{M,m'}^{\Theta_M}$. Therefore 
$\varphi(\Theta_M((u'u)^{-1})m'u'u)=\varphi(\Theta_M(u^{-1})m'u)$. The inner integral becomes $$\int_{U_M}f(xmu'u)\varphi(\Theta_M(u^{-1}m'u))\psi^{-1}(xu'u)dx=\int_{U_M}f(x\Theta_M(u')mu)\varphi(\Theta_M(u^{-1})m'u)\psi^{-1}(xu'u)dx$$
$$=\int_{U_M}f(xmu)\varphi(\Theta_M(u^{-1})m'u)\psi^{-1}(x\Theta_M(u')^{-1}u'u)dx=\int_{U_M}f(xmu)\varphi(\Theta_M(u^{-1})m'u)\psi^{-1}(xu)dx$$ where the last equality follows from the compatibility of $\psi$ and $\Theta_M$. Moreover, as $U_MmU_M$ is closed in $M$ and $f\in C^\infty_c(M;\omega_\sigma)$, the integral converges for any $m\in M$.

One can define partial Bessel integrals on any Levi subgroup $L$ of $M$ by an $F$-involution $\Theta_L$ of $L$ in the same fashion. Denote the Weyl group of $M$ (resp. $L$) by $W(M)$(resp. $W(L)$), and $w_M$(resp. $w_L$) the longest element in $W(M)$(resp. $W(L)$). The following terms are used to study general properties of partial Bessel integrals: 
\begin{itemize}
    \item {\textbf{\textrm{B(M)}.}} Following \cite{Cog08}, the subset of Weyl group elements that supports Bessel functions is given by $B(M)=\{w\in W(M): \alpha\in \Delta_M \ \ s.t. \ \ w\alpha>0 \Rightarrow w\alpha\in \Delta_M \}$. There is a bijection $$B(M)\leftrightarrow \{L: \textrm{ Levi\ of\ standard\ parabolic\ subgroups \ of\  } M\}$$ by
    $w\mapsto L=Z_M(\cap_{\alpha\in\theta^+_{M,w}} \ker\alpha)$, where  $\theta_{M,w}^+=\{\alpha\in \Delta_M: w\alpha>0\}$, and conversely
$L \mapsto w=w_M w_L^{-1}$.

\item
$\textbf{U}_{M,w}^+, \textbf{U}_{M,w}^-.$ For each $w\in W(M)$, define
$$U_{M,w}^+=\{u\in U_M: wuw^{-1}\in U_M\}=U_M\cap w^{-1}U_M w$$
$$U_{M,w}^-=\{u\in U_M:wuw^{-1}\in \overline{U_M}\}=U_M\cap w^{-1}\overline{U_M}w$$ where $\overline{U_M}$ is the opposite of $U_M$. Then $U_M=U^+_{M,w}U^-_{M,w}$. Note that if $w=w_L$, then
$U_{M,w_L}^+=N_L, U_{M,w_L}^-=U_L$; if $w=w_M$, then
$U_{M,w_M}^+=\{e\}, U_{M,w_M}^-=U_M$.
In particular, if $w\in B(M)$ with $\dot{w}=\dot{w}_M\dot{w}_L^{-1}$, then $U_{M,w}^+=U_L:=U_M\cap L$,  and $U_{M,w}^-=N_L$ is the unipotent radical of the standard parabolic subgroup of $M$ with Levi component $L$, therefore $U_{M,w}^+$ normalizes $U_{M,w}^-$. Moreover,
$$w U_{M,w}^+w^{-1}=U_{M,w^{-1}}^+,\ w U_{M,w}^-w^{-1}=\overline{U_{M,w^{-1}}^-}$$
Note that if $w=w_M w_L^{-1}\in B(M)$, so is $w^{-1}$. Suppose $w^{-1}=w_M w_{L'}^{-1}$ for some Levi component $L'$ of some standard parabolic subgroup of $M$, then $w_{L'}=w_Mw_L^{-1} w_M$. Consequently, $U_{M,w^{-1}}^+=U_{L'}$, $U_{M,w^{-1}}^-=N_{L'}$.
\item\textbf{Bessel distance}. For $w,w'\in B(M)$ with $w>w'$ define
$$d_B(w,w')=max\{m: \exists w_i\in B(M)\ \ s.t \ \ w=w_m>w_{m-1}>\cdots>w_0=w' \}.$$

\item \textbf{Bruhat order}. For $w,w'\in W(M)$, $w\leq w'\Longleftrightarrow C(w)\subset \overline{C(w')}$, where $C(w)=U_M T_M w U_{M,w}^{-}$ is the Bruhat cell for $w$. Denote by $\Omega_{w}=\bigsqcup_{w\le w'}C(w')$. Then $\Omega_w$ is open in $M$ and $C(w)$ is closed in $\Omega_w$, by Lemma 5.2 \cite{CST17}. 

\item\textbf{The relevant torus $\textbf{A}_w$}.
For $w\in B(M)$, define $A_w=\{a\in A: a \in \cap_{\alpha\in \theta_{M,w}^+}\ker \alpha\}^\circ\subset A$. Note that it is also the connected center of $L_w=Z_M(\cap_{\alpha\in \theta_{M,w}^+}\ker \alpha)$.

\item \textbf{The relevant Bruhat cell} $\textbf{C}_{r}(\dot{w})$. We call $C_r(\dot{w})=U_M\dot{w}A_wU_{M,w}^-$ the relevant part of the Bruhat cell $C(w)$. Note that $C_r(\dot{w})$ depends on the choice of the representative $\dot{w}$ of $w$. By Proposition 5.1 of \cite{CST17}, we can choose $\dot{w}$ so that it is compatible with $\psi$, i.e., $\psi(\dot{w}u\dot{w}^{-1})=\psi(u)$ for all $u\in U_M$.
\item \textbf{Transverse tori}. For $w,w'\in B(M)$, let $L=L_w$ and $L'=L_{w'}$ be their associated Levi subgroups respectively. Suppose $w'\leq w$. Then $L\subset L'$ and $A_{w'}\supset A_w$. Set $A^{w'}_w=A_w\cap L'_{der}=Z_L\cap L'_{der}$, then $A^{w'}_w\cap A_{w'}=A^{w'}_{w'}=Z_{L'}\cap L'_{der}$ is finite and the subgroup $A^{w'}_w A_{w'}\subset A_w$ is open and of finite index.
\end{itemize}

The following two assumptions are needed in the proof of some general properties of partial Bessel integrals, and we will verify them later in our case.

\begin{assumption}\label{involution}For each Levi component $L$ of a standard parabolic subgroup of $M$, suppose there exists an $F$-involution $\Theta_L$ of $L$ preserving the $F$-splitting $\{B_L:=B_M\cap L, A, \{x_\alpha\}_{\alpha\in \Delta_L}\}$, compatible with $\psi$, i.e., $\psi(\Theta_L(u))=\psi(u)$ for all $u\in U_L$, such that $$\Theta_M(u)_+=\dot{w}\Theta_L(u^+)\dot{w}^{-1} ,\ \forall u\in U_{M,m}^{\Theta_M},$$ where $w=w_Mw_L^{-1}\in B(M)$ with its representative $\dot{w}=\dot{w}^M_L:=\dot{w}_M\dot{w}_L^{-1}$ chosen to be compatible $\psi$, and $\Theta_M(u)=\Theta_M(u)_+\Theta_M(u)_-$ with $\Theta_M(u)_\pm\in U_{M,w^{-1}}^\pm$.
\end{assumption}

\begin{assumption}\label{conjugate-invariance}
     $\varphi$ is invariant under the $\Theta_M$-twisted conjugate action by some open compact subgroup $U_0$ of $U_M$, i.e. $\varphi(\Theta_M(u^{-1})mu)=\varphi(m)$ for all $u\in U_0$.
\end{assumption}

\begin{lem}\label{Twistedcentralizer}
    Let $H\subset L\subset M$ be Levi subgroups of $M$, then for any $m=u_1\dot{w}^{M}_H au_2\in C_r(\dot{w}^M_H)$ satisfying Assumption \ref{involution}, we have
    $$U^{\Theta_{M}}_{M,m}=\Theta_M(u_1^-)\Theta_M(U^{\Theta_{L^w}}_{L^w,\Theta_{L^w}(l^{\dot{w}})})\Theta_M(u_1^-)^{-1}\cap (u_2^-)^{-1} U^{\Theta_L}_{L,l}u_2^-$$
    where $u_1=u_1^-u_1^+$ with $u_1^-\in U_{M,w^{-1}}^-$ and $u_1^+\in U_{M,w^{-1}}^+$, $u_2=u_2^+u_2^-$ with $u_2^+\in U_{M,w}^+$ and $u_2^-\in U_{M,w}^{-}$, $l=\dot{w}^{-1}u_1^+\dot{w} \dot{w}^L_Ha u_2^+\in L$ in which $w=w^M_L:=w_{M}w_L^{-1}\in B(M)$, $\dot{w}=\dot{w}_M\dot{w}_L^{-1}$, $L^w=wL w^{-1}$, $l^{\dot{w}}=\dot{w}l\dot{w}^{-1}$, $\Theta_{L^w}=\mathrm{Ad}(\dot{w})\circ\Theta_L\circ\mathrm{Ad}(\dot{w}^{-1})$ is an $F$-involution on $L^w$ preserving the corresponding splitting.
\end{lem}
\begin{proof} $$\Theta_M(u)^{-1}mu=u\Leftrightarrow (\Theta_M(u)^{-1}u_1^-\dot{w}\Theta_L(u^+) \dot{w}^{-1})\dot{w}(\Theta_L(u^+)^{-1}l u^+)((u^+)^{-1}u_2^-u^+ u^-)=u_1^-\dot{w}l u_2^-.$$
Write $\Theta_M(u)=\Theta_M(u)_-\Theta_M(u)_+$, with $\Theta_M(u)_{\pm}\in U_{M,w^{-1}}^{\pm}$. By assumption \ref{involution}, 
$$\Theta_M(u)^{-1}u_1^-\dot{w}\Theta_L(u^+) \dot{w}^{-1}=\Theta_M(u)_-^{-1}\Theta_M(u)_+^{-1}u_1^-\Theta_M(u)_+\in U_{M,w^{-1}}^-, $$ we also have that
$$\Theta_L(u^+)^{-1}lu^+\in L,\ (u^+)^{-1}u_2^-u^+u^-\in U_{M,w}^{-}$$ By uniqueness of the decomposition $\Omega_w=U_{M,w^{-1}}^-\times \dot{w}L\times U_{M,w}^-$, we obtain that
\begin{equation}\label{(1)}
    \Theta_M(u)^{-1}u_1^-\dot{w}\Theta_L(u^+) \dot{w}^{-1}=u_1^-, 
\end{equation}
\begin{equation}\label{(2)}
    \Theta_L(u^+)^{-1}lu^+=l,
\end{equation}
\begin{equation}\label{(3)}
     (u^+)^{-1}u_2^-u^+u^-=u_2^-. 
\end{equation}
Note that (\ref{(2)}) is equivalent to $u^+\in U_{L,l}^{\Theta_L}$. Let $u_+=\dot{w}u^+\dot{w}^{-1}\in U_{L^w}$, then this is equivalent to $u_+\in U_{L^w,l^{\dot{w}}}^{\Theta_{L^w}}$ where $l^{\dot{w}}=\dot{w}l\dot{w}^{-1}$. On the other hand,
(\ref{(1)}) is equivalent to 
\begin{equation*}\label{(1)'}
u=\Theta_M(u_1^-)\Theta_M(\dot{w}\Theta_L(u^+)\dot{w}^{-1})\Theta_M(u_1^-)^{-1}=\Theta_M(u_1^-)\Theta_M(\Theta_{L^w}(u_+))\Theta_M(u_1^-)^{-1} 
\end{equation*} (\ref{(3)}) is equivalent to \begin{equation*}\label{(3)'}
    u=(u_2^-)^{-1}u^+ u_2^-
\end{equation*} The statement of the lemma follows immediately from these equalities.
\end{proof}


Lemma \ref{Twistedcentralizer} characterizes the relationship between the $\Theta_M$-twisted centralizer of any element $m=u_1^-\dot{w}^M_Ll u_2^-$ in the relevant cell $C_r(\dot{w}^M_H)$ with the $\Theta_L$-twisted centralizer of $l\in L$. This property allows one to deduce the descent formula for partial Bessel integrals.
\begin{lem}\label{Omega_w-surjectivity}Suppose $w=w_Mw_L^{-1}\in B(M)$ for some Levi subgroup $L$ of $M$, then \[\Omega_w\simeq U_{M,w^{-1}}^-\times \dot{w}L\times U_{M,w}^-\] and the decomposition is unique. Moreover, the map \[C^\infty_c(\Omega_w;\omega_\sigma)\rightarrow C^\infty_c(L; \omega_\sigma)\]
\[f\mapsto h_f:=\int_{U_{M,w}^-}\int_{U_{M,w^{-1}}^-}f(x^-\dot{w}lu^-)\psi^{-1}(x^-u^-)dx^-du^-\] is surjective.
\begin{proof} See Lemma 5.9 \cite{CST17}. \end{proof}    
\end{lem}
\begin{prop}\label{Descentformula} Suppose $H\subset L\subset M$ are Levi subgroups of $M$, and $\varphi$ satisfies Assumption \ref{conjugate-invariance}. Let $w=w^M_L\in B(M)$. Then for any fixed $m\in C_r(\dot{w}^M_H)$ satisfying Assumption \ref{involution}, decompose $m$ as $m=u_1^- \dot{w} l u_2^-$, with $u_1^-\in U_{M,w^{-1}}^-$, $u_2^-\in U_{M,w}^-$, $l\in L$, we have 
$$B^M_\varphi(m,f)=\psi(u_1^-u_2^-)B^L_\varphi(u_1^-, u_2^-, l,h_f)$$
where 
\[
B_\varphi^L(u_1^-,u_2^-, l,h_f)=\int_{U_{L,l, n_0}^{\Theta_L}\backslash U_L}\int_{U_L}h_f(x'l u')\varphi(\Theta_M(u')^{-1}\Theta_M(n_0)\dot{w}l' u' )\psi^{-1}(x'u')dx'du'\]
in which $n_0=u_2^-\Theta_M(u_1^-)$, $U_{L, l, n_0}^{\Theta_L}:=n_0 \Theta_M(U_{L^w,\Theta_{L^w}(l^{\dot{w}})}^{\Theta_{L^w}}) n_0^{-1}\cap U_{L.l}^{\Theta_L}$ and $l=zl'$ with $z\in Z_M$.
\end{prop}

\begin{proof}Decompose $x\in U_M$ as $x=x^-x^+$ with $x^\pm\in U_{M,w^{-1}}^\pm$. Since $u_2^-\in U_{M,w}^-\subset U_M$, one can make a change of variable $u\mapsto (u_2^-)^{-1}u u_2^-=(u_2^-)^{-1}u^+u^- u_2^-$ with $u^\pm\in U_{M,w}^\pm$. Set $$n_0:=u_2^-\Theta_M(u_1^-)\in U_M, \ U_{L, l, n_0}^{\Theta_L}:=n_0 \Theta_M(U_{L^w,\Theta_{L^w}(l^{\dot{w}})}^{\Theta_{L^w}}) n_0^{-1}\cap U_{L.l}^{\Theta_L}.$$ Note that $U_{L,l.n_0}^{\Theta_L}\subset U_{L,l}^{\Theta_L}\subset U_L=U_{M,w}^+$.
By Lemma \ref{Twistedcentralizer}, $U_{M,m}^{\Theta_M}=(u_2^-)^{-1}U_{L,l, n_0}^{\Theta_L}u_2^- $.
$$B^M_\varphi(m,f)=\int_{U_{M,m}^{\Theta_M}\backslash U_M}\int_{U_M}f(xmu)\varphi(\Theta_M(u^{-1})m'u)\psi^{-1}(xu)dxdu,$$ 
$$=\int_{U_{L,l, n_0}^{\Theta_L}\backslash U_{M,w}^+}\int_{U_{M,w^{-1}}^+}\int_{U_{M,w}^-}\int_{U_{M,w^{-1}}^-}f(x^-x^+u_1^-\dot{w}l u_2^-(u_2^-)^{-1}u^+u^-u_2^-)$$$$\cdot\varphi(\Theta_M((u_2^-)^{-1}(u^-)^{-1}(u^+)^{-1}u_2^-)u_1^-\dot{w}l' u_2^-(u_2^-)^{-1}u^+u^-u_2^-)\psi^{-1}(x^-x^+ (u_2^-)^{-1}u^+u^-u_2^-)dx^-du^-dx^+du^+$$
Set
$$y^-:=x^-x^+u_1^-(x^+)^{-1}\in U_{M,w^{-1}}^-,\  x':=\dot{w}^{-1}x^+ \dot{w}\in \dot{w}^{-1}U_{M,w^{-1}}^+\dot{w}= U_{M,w}^+,\ u':=u^+\in U_{M,w}^+, v^-:=u^-u_2^-\in U_{M,w}^-,$$
Then the Jacobian of this change of variable is 1 and $$f(x^-x^+u_1^-\dot{w}l u_2^-(u_2^-)^{-1}u^+u^-u_2^-)=f(y^-\dot{w}x'l u' v^-),$$ $$\varphi(\Theta_M((u_2^-)^{-1}(u^-)^{-1}(u^+)^{-1}u_2^-)u_1^-\dot{w}l' u_2^-(u_2^-)^{-1}u^+u^-u_2^-)=\varphi(\Theta_M(v^-)^{-1}\Theta_M(u')^{-1}\Theta_M(n_0)\dot{w}l' u' v^-).$$ By compatibility of $\psi$ and $\dot{w}$, $\psi(x')=\psi(x^+)$. Since $f\in C^\infty_c(\Omega_{w};\omega_\pi)$, $\Omega_w=U_{M,w^{-1}}^-\times \dot{w}L\times U_{M,w}^-$, $U_{M,w^{-1}}^-$ and $U_{M,w}^-$ are closed in $\Omega_w$, there exists compact subgroups $U_1$ of $U_{M,w}^-$ and $U_2$ of $U_{M,w^{-1}}^-$ such that $f(y^-\dot{w}x'lu'v^-)\neq 0$ implies that $y^-\in U_1$ and $v^-\in U_2$. On the other hand, by Assumption \ref{conjugate-invariance}, $\varphi$ is invariant under the $\Theta_M$-twisted conjugate action by some open compact subgroup $U_0$ of $U_M$. Shrink $U_2$ if necessary so that $v^-$ lies in $U_0$, it follows that $\varphi(\Theta_M(v^-)^{-1}\Theta_M(u')^{-1}\Theta_M(n_0)\dot{w}l' u' v^-)=\varphi(\Theta_M(u')^{-1}\Theta_M(n_0)\dot{w}l' u')$. Consequently, the above equality is 
$$=\psi(u_1^-u_2^-)\int_{U_{L,l, n_0}^{\Theta_L}\backslash U_L}\int_{U_L}\int_{U_{M,w}^-}\int_{U_{M,w^{-1}}^-}f(y^-\dot{w}x'l u' v^-)\varphi(\Theta_M(u')^{-1}\Theta_M(n_0)\dot{w}l' u' )\psi^{-1}(x'u')dy^-dv^-dx'du'$$
$$=\psi(u_1^-u_2^-)\int_{U_{L,l, n_0}^{\Theta_L}\backslash U_L}\int_{U_L}h_f(x'l u')\varphi(\Theta_M(u')^{-1}\Theta_M(n_0)\dot{w}l' u' )\psi^{-1}(x'u')dx'du'=\psi(u_1^-u_2^-)B^L_\varphi(u_1^-,u_2^-,l, h_f),$$ where the second equality follows from Lemma \ref{Omega_w-surjectivity}.

\end{proof}

\begin{remark}If we assume that
$\Theta_M(L)=L^w$, then $\Theta_L:=\mathrm{Ad}(\dot{w}^{-1})\circ \Theta_M\vert_L$ defines an $F$-involution on $L$, preserving the splitting on $L$. This assumption implies Assumption \ref{involution}. We provide a simple proof here. For any $u=u^+u^-\in U_M$ with $u^\pm\in U_{M,w}^\pm$, as $\Theta_M(L)=L^w$, and $\Theta_M$ preserves $U_M$, we have $\Theta_M(u^\pm)\in U_{M,w^{-1}}^\pm$. Thus $\Theta_L(u^+)=\dot{w}^{-1}\Theta_M(u^+)\dot{w}\in \dot{w}^{-1}U_{M,w^{-1}}^+\dot{w}=U_{M,w}^+=U_L$. As a result, $\Theta_L$ preserves the splitting on $L$ and \[\Theta_M(u)_+=\Theta_M(u^+)=\dot{w}\Theta_L(u^+)\dot{w}^{-1}\]When $\Theta_M=\mathrm{Ad}(\dot{w}_0^{-1})$, this assumption is equivalent to the condition that $\mathrm{Ad}(\dot{w}^G_L) L=L$. Consequently, $L$ is self-associate in $G$. It is the case in \cite{CST17}, \cite{She19}, and \cite{DS25}.
    
\end{remark}

\begin{remark}  It is worth to mention that Assumption \ref{involution} fails to be true in general, especially when $G$ is of type A. A typical counter example is $G=\mathrm{GL}_6$, $M=\mathrm{GL}_3\times\mathrm{GL}_3$, and $L=\mathrm{GL}_3\times \mathrm{GL}_2\times\mathrm{GL}_1$.   
\end{remark}

\begin{lem}\label{trivial-stabilizer} For $n\in N'$ such that $\dot{w}_0^{-1}n=mn'\overline{n}$, set $\Theta_M=\mathrm{Ad}(\dot{w}_0^{-1})$. Suppose 
$U_{M,n}$ is trivial for generic $n$, then up to a set of measure zero,
Assumption \ref{involution} is satisfied for $m$ being the image of such $n$. 
\end{lem}
\begin{proof} By Lemma \ref{Sundaravaradhan}, up to a set of measure zero, \[U_{M,n}=U_{M,m}':=\{u\in U_M: mum^{-1}\in U_M, \psi(mum^{-1})=\psi(u)\}\]
  is trivial. On the other hand, the compatibility of $\psi$ with $\dot{w}_0$ implies that $U_{M,m}^{\Theta_M}\subset U_{M,m}'$. Hence $u^+=u_+=1$.   
\end{proof}

\begin{remark}The function $B_\varphi^L(u_1^-,u_2^-,l, h_f)$ is a partial Bessel function on $L$ depending on the Bruhat decomposition $m=u_1^-\dot{w}l u_2^-$. In particular, if $n_0=1$ and the condition $\Theta_M(L)=L^w$ is satisfied, then
    $$B^M_\varphi(m, f)=B^L_{\varphi_{\dot{w}}}(l, h_f)$$ where $\varphi_{\dot{w}}(m):=\varphi(\dot{w}m)$. This is the case in \cite{CST17}.
\end{remark}

\begin{cor}\label{wa-descent-formula} Suppose $\Theta_M\vert_{U_M}=1$, for $w=w_Mw_L^{-1}\in B(M)$ and $a\in A_w=Z_L$, we have \[U_{M,\dot{w}a}^{\Theta_M}=U_{L\cap L^w,\dot{w}}:=\{u^+\in U_L\cap U_{L^w}: \dot{w}^{-1}u^+\dot{w}=u^+   \}\] Moreover, Assumption \ref{involution} is satisfied for $m=\dot{w}a, a\in A_{w}$ with any $F$-involution $\Theta_L$ of $L$ preserving the splitting of $L$ such that $\Theta_L\vert_{U_L}=1$, and \[B^M_\varphi(\dot{w}a, f)=B^L_{\varphi_{\dot{w}}}(a, h_f) \]
\end{cor}

\begin{remark} In our case, we may simply take $\Theta_L=1$.    
\end{remark}
\begin{proof}As $\Theta_M=1$, $u\in U_{M,\dot{w}a}^{\Theta_M}\Leftrightarrow u^{-1}\dot{w}au=\dot{w}a\Leftrightarrow u_-^{-1}u_+^{-1}\dot{w}au^+u^-=\dot{w}a\Leftrightarrow u_-^{-1}\dot{w}(\dot{w}^{-1}u_+^{-1}\dot{w})au^+u_-=\dot{w}a$, in which we decompose $u=u^+u^-=u_+u_-$ where $u^\pm\in U_{M,w}^\pm$ and $u_{\pm}\in U_{M,w^{-1}}^\pm$. Since $a\in A_w$, $u^{-1}\dot{w}au\in C_r(\dot{w})\subset \Omega_{w}=U_{M,w^{-1}}^-\times \dot{w}L\times U_{M,w}^-$, and the decomposition is unique. Therefore $u\in U_{M,\dot{w}a}^{\Theta_L}\Leftrightarrow u_-=u^-=1, \&\   \dot{w}^{-1}u_+^{-1}\dot{w}au^+=a$. The last condition is the same as $\dot{w}^{-1}u^+\dot{w}=u^+$, since $a\in A_w=Z_L$ and $u^+\in U_{M,w}^+=U_L$. Moreover, $u_-=u^-=1$ implies that $u=u^+=u_+\in U_L\cap U_{L^w}$. This proves the first claim of the corollary. The second claim of the corollary is trivially satisfied if $\Theta_L\vert_{U_L}=1$, since $\Theta_M(u)_+=u_+=u^+=\dot{w}^{-1}u^+\dot{w}=\dot{w}^{-1}\Theta_L(u^+)\dot{w}$. By Proposition \ref{Descentformula}, we obtain 
$B^M_\varphi(\dot{w}a, f)=B^L_\varphi(1,1, a, h_f)=B^L_{\varphi_{\dot{w}}}(a, h_f)$.
\end{proof}

\subsubsection{Asymptotic expansions of partial Bessel integrals} In this section, we will derive an asymptotic expansion formula for partial Bessel integrals which is crucial for the analytic stability. From now on, we set $\Theta_M=\mathrm{Ad}(\dot{w}_0^{-1})$. Since in our case $\mathrm{Ad}(\dot{w}_0^{-1})(\gamma_0^\vee(t))=\gamma_0^\vee(t^{-1})$, the $\Theta_M$-twisted conjugate action of $Z_M$ on $A$ is just  $a\mapsto \Theta_M(z)^{-1}az=\gamma_0^\vee(t^2)a=z^2a$, if $z=\alpha^\vee(t)=\gamma_0^\vee(t)\in Z_M^0$. Consequently, if $m=u_1\dot{w}_M au_2$ with $u_1,u_2\in U_M$, $a\in A$, then $m'= u_1 \dot{w}_M a' u_2$ with $a=z^2a'$. We denote $A'$(resp. $A_{w}'$) the subtorus of $A$ (resp. $A_w$) obtained by strpping off the square of the center of $M$, i.e., $A=Z_M^2A'$. (resp. $A_w=Z_M^2 A_w'$.)

By Lemma \ref{w_0-action-on-roots}, if we set $\Theta_M=\mathrm{Ad}(\dot{w}_0^{-1})$, then $\Theta_M\vert_{U_M}=1$.
By Corollary \ref{wa-descent-formula}, for $w=w_Mw_L^{-1}$ and $a\in A_w=Z_L$, $U_{M,\dot{w}a}^{\Theta_M}\subset U_L=U_{M,w}^+$, therefore
\[B^M_\varphi(\dot{w}a, f)=\int_{U_{M,\dot{w}a}^{\Theta_M}\backslash U_{M,w}^+}\int_{U_{M,w}^-}\int_{U_M}f(x\dot{w}au^+u^-)\varphi(\Theta_M(u^-)^{-1}\Theta_M(u^+)^{-1}\dot{w}a'u^+u^-)\psi^{-1}(xu^+u^-)dx du^-du^+.\]
As $au^+=u^+a$, we have $x\dot{w}au^+u^-=x\dot{w}u^+\dot{w}^{-1}\dot{w}au^-$. Make the change of variable $x\mapsto x(\dot{w}u^+\dot{w}^{-1})^{-1}$, and note that $f(x\dot{w}a u^-)\neq 0$ implies that $u^-\in U_2$ for some open compact subgroup $U_2$ of $U_M$. By Assumption \ref{conjugate-invariance}, the above expression
\[=\int_{U_{L\cap L^w,\dot{w}}\backslash U_{M,w}^+}\int_{U_{M,w}^-}\int_{U_M}f(x\dot{w}a u^-)\varphi(\Theta_M(u^+)^{-1}\dot{w}a'u^+)\psi^{-1}(xu^-)dx du^- du^+=\widetilde{\varphi}^{M}_L(a')B^M(\dot{w}a, f)\] where $B^M(\dot{w}a, f)$ is the full Bessel function and 
\[\widetilde{\varphi}^M_L(a')=\int_{U_{L\cap L^w,\dot{w}}\backslash U_{M,w}^+}\varphi(\Theta_M(u^+)^{-1}\dot{w}a' u^+)du^+=\int_{U_{L\cap L^w,\dot{w}}\backslash U_L}\varphi_{\dot{w}}(\dot{w}^{-1}(u^+)^{-1}\dot{w}a' u^+)du^+\] in which $\varphi_{\dot{w}}(l)=\varphi(\dot{w}l)$.

\paragraph{Cutoff convention.}
For each $\kappa$, let $\varphi=\varphi_\kappa:=\mathbf{1}_{X_\kappa}$, where
$X_\kappa\subset M$ is open.  We assume that the family
$\{X_\kappa\}_\kappa$ is compactly exhaustive, in the sense that
\begin{equation}\tag{CE}\label{eq:compact-exhaustion}
  \text{for every compact subset $C\subset M$, there exists $\kappa_C$
  such that }
  C\subset X_\kappa
  \quad\text{for all }\kappa\geq \kappa_C .
\end{equation}
In the notation used above, one may take $X_\kappa=M_{x,\kappa}$.

\begin{lem}\label{lem:cutoff-positive}
Let
\[
  w=w_Mw_L^{-1}\in B(M),\qquad
  a=z^2a',\quad z\in Z_M,\quad a'\in A_w',
\]
where $A_w=Z_M^2A_w'$, and put $H_{\dot{w}}:=U_{L\cap L^w,\dot w}$. If $\dot w a'\in X_\kappa$, then
\[
  \widetilde{\varphi}^{M}_{L}(a')>0.
\]
Consequently, for every compact subset $K\subset A'$, there exists
$\kappa_K$ such that
\[
  \widetilde{\varphi}^{M}_{L}(a')>0
  \qquad
  \text{for all }a'\in K\text{ and all }\kappa\geq\kappa_K .
\]
\end{lem}

\begin{proof}
Since $\Theta_M|_{U_M}=\operatorname{id}$, the factor occurring above is
\[
  \widetilde{\varphi}^{M}_{L}(a')
  =
  \int_{H_{\dot{w}}\backslash U_L}
  \mathbf{1}_{X_\kappa}
  \bigl(u^{-1}\dot w a'u\bigr)\,du .
\]
Consider the map
\[
  T_{a'}:H_{\dot{w}}\backslash U_L\longrightarrow M,
  \qquad
  H_{\dot{w}}u\longmapsto u^{-1}\dot w a'u .
\]
This map is well defined.  Indeed, if $h\in H_{\dot{w}}$, then
$\dot w^{-1}h\dot w=h$, while $a'\in Z_L$ commutes with $h$, and hence
\[
  (hu)^{-1}\dot w a'(hu)
  =
  u^{-1}h^{-1}\dot w a'hu
  =
  u^{-1}\dot w a'u .
\]
The map $T_{a'}$ is continuous and
\[
  T_{a'}(H_{\dot{w}})=\dot w a'.
\]
Thus, if $\dot w a'\in X_\kappa$, the set
$
  T_{a'}^{-1}(X_\kappa)
$
is a nonempty open neighborhood of the identity coset in
$H_w\backslash U_L$.  It therefore has positive quotient measure, and
\[
  \widetilde{\varphi}^{M}_{L}(a')
  =
  \operatorname{vol}\bigl(T_{a'}^{-1}(X_\kappa)\bigr)>0.
\]

Finally, if $K\subset A'$ is compact, then $\dot wK$ is compact in $M$.
By \eqref{eq:compact-exhaustion}, there exists $\kappa_K$ such that
$\dot wK\subset X_\kappa$ for every $\kappa\geq\kappa_K$.  The first
assertion then gives the required uniform positivity.
\end{proof}

\begin{lem}\label{lem:partial-full-nonvanishing}
Let $w=w_Mw_L^{-1}\in B(M)$ and let $a\in A_w=Z_L$.  Given
$f\in C_c^\infty(\Omega_w;\omega_\sigma)$, suppose that
$\varphi=\varphi_\kappa=\mathbf{1}_{X_\kappa}$ satisfies Assumption \ref{conjugate-invariance} and
\eqref{eq:compact-exhaustion}.  Then there exists $\kappa_0$, depending
only on $f$ and $w$, such that, for every $\kappa\geq\kappa_0$,
\[
  B_{\varphi}^{M}(\dot w a,f)\neq 0
  \quad\Longleftrightarrow\quad
  B^{M}(\dot w a,f)\neq 0 .
\]
\end{lem}

\begin{proof}
Using
\[
  C(w)\simeq U_M\times Z_M\times A'\times U_{M,w}^{-},
\]
the compact support of $f$ modulo $Z_M$, hence also compactly supported modulo $Z_M^2$, as $Z_M^2$ is of finite index in $Z_M$, gives compact subsets
\[
  U_1\subset U_M,\qquad
  U_2\subset U_{M,w}^{-},\qquad
  K'\subset A'
\]
such that
\[
  f(x\dot w a u)\neq 0
\]
implies
\[
  x\in U_1,\qquad u\in U_2,\qquad
 a=z^2a'\quad\text{with }z\in Z_M,\ a'\in K'.
\]
By Lemma~\ref{lem:cutoff-positive}, after enlarging $\kappa$ if
necessary, we have
\[
  \widetilde{\varphi}^{M}_{L}(a')>0
  \qquad\text{for every }a'\in K'.
\]

Recall that
\[
  B_{\varphi}^{M}(\dot w a,f)
  =
  \widetilde{\varphi}^{M}_{L}(a')\,
  B^{M}(\dot w a,f).
\]
If $B^{M}(\dot w a,f)\neq0$, then necessarily $a'\in K'$, so the
coefficient on the right-hand side is positive and
$B_{\varphi}^{M}(\dot w a,f)\neq0$.  The reverse implication follows
immediately from the same identity.
\end{proof}

\begin{remark}\label{rem:translated-cutoffs}
No converse is asserted in Lemma~\ref{lem:cutoff-positive}: a twisted
$U_L$-orbit may meet $X_\kappa$ even when its distinguished point
$\dot w a'$ does not lie in $X_\kappa$.

The compact-exhaustion property is preserved under the restrictions and
fixed left translations used in the induction.  More precisely, if
$J\subset M$ is a Levi subgroup and $\dot v\in M$ is fixed, then
\[
  \varphi_{\kappa,\dot v}^{J}(j):=\varphi_\kappa(\dot vj),
  \qquad j\in J,
\]
is the characteristic function of
\[
  X_{\kappa,\dot v}^{J}
  :=
  J\cap\dot v^{-1}X_\kappa .
\]
For every compact $C\subset J$, the set $\dot vC$ is compact in $M$;
hence $C\subset X_{\kappa,\dot v}^{J}$ for all sufficiently large
$\kappa$.  Therefore Lemma~\ref{lem:partial-full-nonvanishing} applies
at each Levi step occurring in the proof of Proposition \ref{Asymptotic-formula-0}.
\end{remark}


\paragraph{Compatibility with the asymptotic expansion.}
The preceding construction prescribes $\varphi_\kappa$ on
$\mathcal{N}(N')$.  We choose its extension to $M$ so that
$\varphi_\kappa=\mathbf{1}_{X_\kappa}$ for a family satisfying
\eqref{eq:compact-exhaustion}, and so that the same property holds for
the fixed left translates restricted to the Levi subgroups occurring
in Proposition \ref{Asymptotic-formula-0}.  This choice does not change the identity
\[
  j_{\sigma,\kappa}(\dot n)
  =
  B_{\varphi_\kappa}^{M}(\dot m,f),
\]
because all values of $\varphi_\kappa$ entering that identity are on
$\mathcal{N}(N')$.

\begin{remark}Replace Lemma 5.7 and Lemma 5.8  in \cite{CST17}  by Lemma \ref{lem:cutoff-positive} and Lemma \ref{lem:partial-full-nonvanishing} respectively, the proofs of Lemma 5.12, 5.13, and 5.14 in \cite{CST17} hold verbatim in our setting.   
\end{remark}

\begin{prop}\label{Asymptotic-formula-0} Assume $\varphi$ satiefies Assumption \ref{conjugate-invariance}.
Fix an auxiliary function $f_0\in C^{\infty}_c(M;\omega_{\pi})$ with $W_{f_0}(e)=1$. Then for each $f\in C_c^{\infty}(M;\omega_{\pi})$ with $W_f(e)=1$ and $w'\in B(M)$ with $d_B(e,w')=1$, there exists a function $f_{w'}\in C^{\infty}_c(\Omega_{w'};\omega_{\pi})$ such that for any $w\in B(M)$ and $m\in M$ satisfying Assumption \ref{involution}, we have 
$$B^M_{\varphi}(m,f)=B^M_\varphi(m,f_1)+\sum_{w'\in B(M), d_B(w',e)=1}B^M_{\varphi}(m,f_{w'})$$ where $f_1(m):=\sum_{m=m_1c}f_0(m_1)B^M(\dot{e}c, f)=\sum_{m=m_1c}f_0(m_1)\omega_{\pi}(c)$, the sum runs over all possible decompositions $m=m_1c$ with $m_1\in M_{D}$, and $c\in A_e=Z_M$.
\end{prop}
\begin{proof}
     Write $m=m_1c$ with $m_1\in M_{D}$ and $c\in A_e=Z_M$, then there are only finitely many such decompositions, indexed by elements in the transverse torus $A^e_e=M_{D}\cap Z_M$. 
Define
$$f_1(m):=\sum_{m=m_1c}f_0(m_1)B^M(\dot{e}c, f)=\sum_{m=m_1c}f_0(m_1)\omega_{\pi}(c)$$ Then $f_1(m)\in C_c^{\infty}(M;\omega_{\pi})$. For $a\in A_e=Z_M$, $U_{M,a}^{\Theta_M}=U^{\Theta_M}_{M,e}$, we have
$$B^M_\varphi(\dot{e}a,f_1)=\omega_\pi(a)\int_{U_{M,e}^{\Theta_M}\backslash U_M}\int_{U_M} f_1(xu)\varphi(\Theta_M(u^{-1})a' u)\psi^{-1}(xu)dx du$$$$=\omega_\pi(a)W_{f_1}(e)\int_{U_{M,e}^{\Theta_M}\backslash U_M} \varphi(\Theta_M(u^{-1})a'u)du.$$ 
As $W_{f_1}(e)=\int_{U_M} f_1(x)\psi^{-1}(x)dx$, while $x\in U_M\subset M_{D}$, so $f_1(x)=f_0(x)$. Thus $W_{f_1}(e)=W_{f_0}(e)=1$, and we obtain
 $B^M_\varphi(\dot{e}a, f-f_1)=0$ for all $a\in A_e$. Apply \cite[Lemma 5.13]{CST17}, there exists $f_2'\in C_c^\infty(\Omega_e^\circ,;\omega_\pi)$, where $\Omega_e^\circ=\Omega_e-C(e)=M-B_M$, such that $B^M_\varphi(m,f-f_1)=B^M_\varphi(m, f_2')$ for all $m\in M$ . Let $\Omega_1=\cup_{w\in B(M), w\neq e}\Omega_w=\cup_{w\in B(M), d_B(w,e)=1}\Omega_w$, and $\Omega_0=M-C(e)=\Omega^\circ_e$. By \cite[Lemma 5.14]{CST17}, there exists $f_2\in C^\infty_c(\Omega_1;\omega_\pi)$, so that $B^M_\varphi(m, f_2)=B^M_\varphi(m, f_2')=B^N_\varphi(m, f-f_1)$ for sufficiently large $\varphi$ depending only on $f_1$. Finally, by partition of unity, for each $w'\in B(M)$ with $d_B(w',e)=1$, we can find $f_{w'}\in C^\infty_c(\Omega_{w'};\omega_\pi)$, such that $f_2=\sum_{w'\in B(M), d_B(w',e)=1}f_{w'}$. 
\end{proof}

\begin{prop}\label{Asymptoticformula} Assume $\varphi$ satisfies Assumption \ref{conjugate-invariance}.
    Fix an auxiliary function $f_0\in C^{\infty}_c(M;\omega_{\sigma})$ with $W_{f_0}(e)=1$. Then for each $f\in C_c^{\infty}(M;\omega_{\sigma})$ with $W_f(e)=1$ and $w'\in B(M)$ with $d_B(e,w')\ge 1$, there exists a function $f_{w'}\in C^{\infty}_c(\Omega_{w'};\omega_{\sigma})$ such that for any $m=u_1\dot{w}au_2\in C_r(\dot{w})$ satisfying Assumption \ref{involution} with $\dot{w}=\dot{w}^M_L\in B(M)$,
$$B^M_{\varphi}(m,f)=\sum_{a=bc}\omega_\sigma(c)B^M_\varphi(u_1 \dot{w}bu_2, f_0)+\sum_{w'\in B(M), d_B(w',e)\ge 1}B^M_{\varphi}(m,f_{w'})$$ 
where $a=bc$ runs over the possible decompositions of $a\in A_w$ with $b\in A^e_w$ and $c\in A_e=Z_M.$
\end{prop}

\begin{proof}
    This is Proposition 6.7 in \cite{DS25}, but the proof needs a slight modification. For the convenience of the reader, we provide a complete proof here. By Lemma \ref{Twistedcentralizer}, $U^{\Theta_M}_{M,m}\subset u_2^{-1}U_{M,w}^+u_2$. On the other hand, $$U_M=u_2^{-1}U_{M,w}^+ U_{M,w}^- u_2^{-1}=u_2^{-1}U_{M,w}^+ u_2(u_2^{-1}U_{M,w}^-u_2).$$ It follows that if we write $u=u'(u_2^{-1}u^-u_2)$ where $u'=u_2^{-1}u^+u_2$ with $u^+\in U_{M,w}^+$, $u^-\in U_{M,w}^-$, then
 $$B^M_{\varphi}(m,f_1)=\int_{U_{M,m}^{\Theta_M}\backslash u_2^{-1}U_{M,w}^+u_2}\int_{U_{M,w}^-}\int_{U_M}f_1(xu_1\dot{w}au_2u'u_2^{-1}u^-u_2)$$$$\cdot\varphi(\Theta_M(u_2^{-1}(u^-)^{-1}u_2 {u'}^{-1})u_1\dot{w} a' u_2 u' u_2^{-1}u^-u_2)\psi^{-1}(xu' u_2^{-1}u^-u_2)dxdu^-du'.$$
 As $u^+=u_2u'u_2^{-1}\in U_{M,w}^+$ and $a\in A_{w}$, we have $xu_1\dot{w}au_2^{-1}u'u_2^{-1}u^-u_2=xu_1(\dot{w}u_2u'u_2^{-1}\dot{w}^{-1})\dot{w}au^-u_2$. Let $x'=xu_1(\dot{w}u_2u'u_2^{-1}\dot{w}^{-1})\in U_{M,w}^+$, $v^-=u^-u_2\in U_{M,w}^-$. By compatibility of $\psi$ with $\dot{w}$, $\psi(x')=\psi(x)\psi(u_1)\psi(u')$, thus
$$B^M_\varphi(m,f_1)=\psi(u_1u_2)\int_{U_{M,m}^{\Theta_M}\backslash u_2^{-1}U_{M,w}^+u_2}\int_{U_{M,w}^-}\int_{U_M}f_1(x'\dot{w} av^-)$$$$\cdot\varphi(\Theta_M((v^-)^{-1}u_2 {u'}^{-1})u_1\dot{w} a' u_2 u' u_2^{-1}v^-)\psi^{-1}(x' v^-)dx'dv^-du'$$
 By the construction of $f_1$, decompose $x'\dot{w}av^-=m_1c$ with $m_1\in M_{D}$ and $c\in A_e=Z_M$, which is equivalent to $$ac^{-1}=\dot{w}^{-1}(x')^{-1}m_1(v^-)^{-1}.$$ Since we pick $\dot{w}\in M_{D}$, and $x', v^-\in U_M\subset M_{D}$, this is saying that $b:=ac^{-1}\in M_{D}\cap A_{w}=A_{w}^e$. It follows that $$f_1(x'\dot{w}av^-)=\sum_{a=bc}f_0(x'\dot{w}bv^-)\omega_\pi(c),$$ and consequently
$$B^M_\varphi(m,f_1)=\sum_{a=bc}\omega_\pi(c)B^M_\varphi(u_1 \dot{w}bu_2, f_0)$$ where $a=bc$ runs over the possible decompositions of $a\in A_w$ with $b\in A^e_w$ and $c\in A_e=Z_M.$

 For $w'=w^M_L\in B(M)$, let $h_{w'}=h_{f_{w'}}\in C^\infty_c(L;\omega_\pi)$ be the image of $f_{w'}$ under the surjective map $C^\infty_c(\Omega_{w'}; \omega_\pi)\twoheadrightarrow C^{\infty}_c(L;\omega_\pi)$. Pick $h_0\in C^\infty_c(L;\omega_\pi)$ normalized so that $B^L_{\varphi_{\dot{w}'}}(\dot{e}, h_0)=\frac{1}{\vert Z_L\cap A^{w'}_{w'}\vert}$, and $B^M(b, h_0)=0$ for $b\in A^{w'}_{w'}$ but $b\notin Z_M\cap A^{w'}_{w'}$. Similar to the construction of $f_1$ from $f$, let \[h_{1,w'}(l)=\sum_{l=l_1c}h_0(l_1)B^L(c,h_{w'}),\] the sum is taken over all possible decompositions $l=l_1c$ with $l_1\in L_{D}$ and $c\in A_{w'}=Z_L$. Apply Lemma \ref{lem:partial-full-nonvanishing} to the Levi subgroup $L$, the proof of \cite[Proposition 5.4]{CST17} also carries over verbatim, thus 
$B^L_{\varphi_{\dot{w}'}}(a, h_1)=B^L_{\varphi_{\dot{w}'}}(a, h_{w'})$ for all $a\in A_{w'}=Z_L$. Choose $f_{1,w'}$ which maps to $h_{1,w'}$ under $C^{\infty}_c(\Omega_{w'}, \omega_{\pi}) \twoheadrightarrow C^{\infty}_c(L;\omega_{\pi})$, and apply Corollary \ref{wa-descent-formula}, we obtain
$B^M_{\varphi}(w'a,f_{1,w'})=B^L_{\varphi_{\dot{w}'}}(a, h_{1,w'})=B^L_{\varphi_{\dot{w}'}}(a,h_{w'})=B^M_\varphi(w'a,f_{w'})$. So $B^M_\varphi(w'a, f_{w'}-f_{1,w'})=0$, for all $a\in A_{w}$. By \cite[Lemma 5.13, 5.14]{CST17} again, together with partition of unity, there exists $f_{w',w''}\in C^\infty_c(\Omega_{w''};\omega_\pi)$ for $w'<w''\in B(M)$ such that $d_B(w',w'')=1$, and
$$B^M_{\varphi}(m, f_{w'})=B^M_{\varphi}(m,f_{1,w'})+\sum_{w''\in B(M), w''>w', d_B(w'',w')=1}B^M_{\varphi}(m,f_{w',w''})$$ for sufficiently large $\varphi$. Proceed by repeating this process and do induction on the Bessel distance $d_B$, together with the expansion in Proposition \ref{Asymptotic-formula-0}, we obtain the asymptotic expansion formula as desired.
\end{proof}

\subsubsection{Uniform smoothness} In this subsection we prove certain uniform smoothness of partial Bessel functions in the second part of the asymptotic expansion formula in Proposition \ref{Asymptoticformula}.

\begin{defn} A function $B$ on a p-adic group $H$ is \textbf{uniform smooth} if it is uniformly locally constant, i.e. for each $h\in H$, there exists an open compact subgroup $K_0\subset H$ independent of $h\in H$, such that
\[B(h k)=B(h), \ \forall k\in K_0\]    
\end{defn}

Recall for each $w'=w_M w_L^{-1}\in B(M)$, we associated to $h_{w'}\in C^\infty_c(L;\omega_\pi)$ the function 
\[h_{1,w'}=\sum_{l=l_1c}h_0(l_1)B^L(c,h_{w'})\] with $h_0\in C^\infty_c(L;\omega_\pi)$ chosen as in the proof of Proposition \ref{Asymptoticformula}. Suppose $m=u_1^-\dot{w}'lu_2^-$, where $l=v_1\dot{w}a v_2\in C^L(w)$, a Bruhat cell in $L$ with $z\in Z_M$ with $\dot{w}\in L_{der}$. Write $l=z^2l'$ with $z\in Z_M$, then
$$B^L_\varphi(u_1^-,u_2^-,l',h_{1,w'})=\int_{U^{\Theta_L}_{L,l,n_0}\backslash U_L}\int_{U_L}h_{1,w'}(xl'u)\varphi(\Theta_M((n_0u)^{-1})\dot{w}'l'u)\psi^{-1}(xu)dxdu.$$ Decompose $xl'u=l_1c$ with $l_1\in L_{der}$, $c\in Z_L$. By construction, $l'=v_1\dot{w}a'v_2$ where $a=z^2a'$. As $x, u, v_1, v_2,\dot{w}\in L_{der}$, the decomposition $xl'u=l_1c$ is equivalent to $a'=bc$ with $b\in A\cap L_{der}$ and $c\in Z_L$.
Let $b=z_b^2b'$, $c=z_c^2c'$ with  $z_b,z_c\in Z_M$, $b'\in A'$, $c'\in Z_L'$, then $a'=(z_bz_c)^2b'c'$ implies that $(z_bz_c)^2=1$ and $a'=b'c'$. Consequently, 
\[h_{1,w'}(xl'u)=\sum_{a'=b'c'}h_0(xl_{b'}u)B^L(c' , h_{w'})\] where $l_{b'}= v_1\dot{w} b' v_2$. Recall that $U_{L, l, n_0}^{\Theta_L}:=n_0 \Theta_M(U_{L^w,\Theta_{L^w}(l^{\dot{w}})}^{\Theta_{L^w}}) n_0^{-1}\cap U_{L.l}^{\Theta_L}$ with  $n_0=u_2^-\Theta_M(u_1^-)$. Since $U_{L,l'}^{\Theta_L}=\{u\in U_L:\Theta_L(u^{-1})l_{b'}uc'=l_{b'}c'\}=\{u\in U_L:\Theta_L(u^{-1})l_{b'}u=l_{b'}\}=U_{L, l_{b'}}^{\Theta_L}$, similarly we have $U^{\Theta_{L^w}}_{L^w, \Theta_{L^w}(l'^{w'})}=U^{\Theta_{L^w}}_{L^w, \Theta_{L^w}(l_{b'}'^{w'})}$. Thus $U^{\Theta_L}_{L,l',n_0}=U^{\Theta_L}_{L, l_b',n_0}$ and
$$B^L_\varphi(u_1^-,u_2^-,l', h_{1,w'})=\sum_{a'=b'c'}B^L(c', h_{w'})\int_{U_{L,l'_{b'},n_0}^{\Theta_L}\backslash U_L}\int_{U_L}h_0(xl_{b'}u)$$$$\cdot\varphi(\Theta_M(u)^{-1}\Theta_M(n_0)l_{b'}uc')\psi^{-1}(xu)dxdu=\sum_{a'=b'c'}B^L(c',h_{w'})B^L_{\varphi^{c'}}(u_1^-,u_2^-, l_{b'}, h_0)$$ where $\varphi^{c'}(\cdot):=\varphi(\cdot c')$ is the right shift of $\varphi$ by $c'$.

Let $\mathrm{Pr}_A: C(w_M)\simeq U_M\times w_MA\times U_{M,w_M}^-\simeq U_M\times w_M A\times U_M\rightarrow A$ be the projection map onto the maximal split torus $A$. For each $w'\in B(M)$, set \[A_{R,w'}:=\mathrm{Pr}_A(\mathrm{Im}\mathcal{N})\cap A_{w'}, \qquad A_{R,w'}':=\mathrm{Pr}_A(\mathrm{Im}\mathcal{N})\cap A_{w'}'.\] Let $A_{R',w'}$ be the subtorus obtained in the same way as $A_{R,w'}$ by setting $c_5=1$. Then by the $Z_M^0$-equivariance of the map $\mathcal{N}:\dot{n}\mapsto \dot{m}$ via $\dot{w}_0^{-1}\dot{n}=\dot{m}\dot{n}'\overline{\dot{n}}$ for the typical representative $\dot{n}=\prod_{i=0}^4x_{\gamma_i}(c_i)$  of $ Z_M^0U_M\backslash N'$ with $\underline{c}=(c_0,\cdots, c_4)\in R'$, we can identify $A'_{R,w'}$ with $A_{R',w'}$.
\begin{prop} \label{uniform-smooth}Assume $\varphi$ satisfies Assumption \ref{conjugate-invariance}. Let $ L\subset M$ be a Levi subgroup of $M$, and $m=\mathcal{N}(n)$ where $n=\prod_{i=0}^5 x_{\gamma_i}(c_i)$ with $R$ the chosen set of orbit space representatives of $U_M\backslash N'$. Then $m$ satisfies Assumption \ref{involution}. 
Decompose $m=u_1\dot{w}_Ma u_2=u_1^-\dot{w}'lu_2^-$ as before, and suppose $a\in A^{w'}_{w_M} A_{w'}\subset A$ with $ w'=w_M w_L^{-1}\in B(M)$ and $d_B(w',e)\ge 1$. Write $a=bc$ with $b\in A^{w'}_{w_M}$, $c\in A_{w'}$. Decompose $A_{w'}=Z_M^2A'_{w'}$ and write $c=z^2c'$ with $z\in Z_M$, $c'\in A'_{w'}$. Then
$$B^M_\varphi(m,f_{1,w'})=\omega_\pi(z^2)\psi(u_1^-u_2^-)\sum_{a'=b'c'}B^L(c', h_{w'})B^L_{\varphi^{c'}}(u_1^-, u_2^-,l_{b'}, h_0)$$ is uniformly smooth as a function of $c'\in A_{R,w'}'=A_{R',w'}$.
\end{prop}
\begin{proof}First of all, $U_{M,n}=1$ by Theorem \ref{orbit-space-representative}. Then Assumption \ref{involution} is satisfied by Lemma \ref{trivial-stabilizer}.
Proposition \ref{Descentformula} and the above argument imply that 
\[B^M_\varphi(m, f_{1,w'})=\psi(u_1^-u_2^-)B^L_\varphi(u_1^-,u_2^-, l,h_{1,w'})=\psi(u_1^- u_2^-)\sum_{a=bc}B^L(c,h_{w'})B^L_{\varphi^c}(u_1^-,u_2^-, l_b, h_0)\] Fix a decomposition $a=bc$ with $b\in A_{w_M}^{w'}$, $c\in A_{w'}$, then all such decompositions are of the form $a=(b\xi^{-1})(\xi c)$ with $\xi$ runs through the finite set $\in A_{w}^{w'}\cap A_{w'}=A_{w'}^{w'}=Z_L\cap L_{der}$. Therefore $\vert\xi\vert=1$, $\varphi^{\xi c}=\varphi^c$ since $\varphi$ depends only on the absolute value. Consequently, $$B^L_\varphi(u_1^-,u_2^-, l, h_{1,w'})=\sum_{\xi\in A_{w'}^{w'}}B^L(\xi c, h_{w'})B^L_{\varphi^{c}}(u_1^-, u_2^-, l_{b\xi^{-1}}, h_0)$$ Write $\xi=\xi_1\xi'$ with $\xi_1\in Z_M$ and $\xi'=\xi \xi_1^{-1}\in A_{w'}'=Z_L'$, then
 $$B^L(\xi c, h_{w'})=\int_{U_M}h_{w'}(x\xi c)\psi^{-1}(x)dx=\omega_{\pi}(\xi_1 z^2)\int_{U_M}h_{w'}(x\xi' c')\psi^{-1}(x)dx$$ The small Bruhat cell $C^L(e_L)=AU_L=Z_M^2 A' U_L$ of $L$ is closed in $L$, thus $Z_L'\subset A'$ is also closed in $L$.
 Since $h_{w'}\in C^\infty_c(L;\omega_\pi)$ is smooth of compact support modulo $Z_M$, there exists compact subsets $K_1\subset Z_L'$ and $U_1\subset U_M$ such that $h_{w'}(x\xi'c')\neq 0$ only if $ \xi'c'\in K_1$ and $x\in U_1$. It follows that the support of \[B^L_\varphi(u_1^-, u_2^-, l,h_{1,w'})=\omega_\pi(z^2)\sum_{\xi\in A_{w'}^{w'}}B^L(\xi c', h_{w'})B^L_{\varphi^{zc'}}(u_1^-, u_2^-, l_{b\xi^{-1}}, h_0)\]
  in $c'\in A_{w'}'=Z_L'$ is compact and independent of the decomposition $a=bc=z^2bc'$, since the right hand side of the above equality does not vanish unless $c'\in K:=\bigcup_{\xi'}\xi'^{-1}K_1$.
  
  On the other hand, by Proposition \ref{BruhatDecomp} and Corollary \ref{finiteetale1},
the induced map $\mathcal{N}$ on $U_M\backslash N'$ is \'etale of degree 2 onto its image. The image $m=\mathcal{N}(n)$ for $n\in R$ lies in the big cell $C(w_M)$ and we have
$$m=u_1 t\dot{w}_M u_2$$ with $t$ of the form $$t=\alpha_1^\vee(t_1)\alpha_2^\vee(t_2)\alpha_3^\vee(t_3)\alpha_4^\vee(t_4)\alpha_5^\vee(t_3)\alpha_6^\vee(t_1).$$ Write $m=u_1t\dot{w}_M u_2=u_1\dot{w}_M au_2$, then $$a=\dot{w}_M^{-1}t\dot{w}_M=\alpha_1^\vee(\frac{t_2}{t_1})\alpha_2^\vee(t_2)\alpha_3^\vee(\frac{t_2^2}{t_3})\alpha_4^\vee(\frac{t_2^3}{t_4})\alpha_5^\vee(\frac{t_2^2}{t_3})\alpha_6^\vee(\frac{t_2}{t_1}).$$ Set \[a_1=\frac{t_2}{t_1}, a_2=t_2, a_3=\frac{t_2^2}{t_3}, a_4=\frac{t_2^3}{t_4}, a_5=\frac{t_2^2}{t_3},a_6=\frac{t_2}{t_1}.\] This shows that $\mathrm{Pr}_A(\mathrm{Im}\mathcal{N})$ is a 4-dimensional subtorus of the form
$$\{a=\alpha_1^\vee(a_1)\alpha_2^\vee(a_2)\alpha_3^\vee(a_3)\alpha_4^\vee(a_4)\alpha_5^\vee(a_3)\alpha_6^\vee(a_1): a_i\in \mathbb{G}_m, i=1,2,3,4\}.$$ In particular, $\mathrm{Pr}_A(\mathrm{Im}\mathcal{N})$ is closed in $A$. Moreover, by Section \ref{gamma_0(w_M)}, $A_M=\{\gamma_0^\vee(t): t\in \mathbb{G}_m\}$, therefore $A_M\cap \mathrm{Pr}_A(\mathrm{Im}\mathcal{N})\neq \emptyset$. Since by Bruhat order, $Z_M=A_M=A_{e}\subsetneq A_{w'}$ for all $w'\in B(M)$, $d_B(w',e)\ge 1$, and $A'_{w'}$ is obtained by depriving the groups of square elements in the center of $M$, i.e., $A_{w'}=Z_M^2 A_{w'}'$, it follows that $A_{R',w'}=A'_{R,w'}=\mathrm{Pr}_A(\mathrm{Im}\mathcal{N})\cap A_{w'}'\neq \emptyset $ for all $w'\in B(M)$ with $d_B(w',e)\ge 1$.  Consequently, the support of $B^L_
\varphi(u_1^-,u_2^-, l, h_{1,w'})$ in $c'\in A'_{R,w'}$ is non-empty and compact for every $w\in B(M)$ with $d_B(w',e)\ge 1$. Moreover, since $h_{w'}$ is smooth, one can find a non-trivial open compact subgroup $C_0\subset A'_{R,w'}$ such that $h_{w'}(x\xi'c'c_0)=h_{w'}(x\xi' c)$ for all $x\in U_1, c'\in A'_{R,w'}$, and $ c_0\in C_0$. Proposition \ref{BruhatDecomp} also implies that the 1-parameter subgroups in $u_1, u_2$, hence $u_1^-$ and $u_2^-$, are rational functions of $(c_0,c_1,c_2,c_3,c_4,c_5)$ without singularities. By Corollary \ref{finiteetale1}, the map $(c_0,c_1,c_2,c_3,c_4,c_5)\mapsto(t_1,t_2,t_3,t_4,t_5,t_6)\mapsto(a_1,a_2,a_3,a_4,a_5,a_6)\in \mathrm{Pr}_A(\mathrm{Im}\mathcal{N})\times \mathbb{G}_m^2$ is \'etale of degree 2. It follows that $u_1^-$ and $u_2^-$ are smooth functions on $\mathrm{Pr}_A(\mathrm{Im}\mathcal{N})\times \mathbb{G}_m^2$ without singularities. Consequently, one can choose $C_0$ such that $u_1^-,u_2^-$ are uniformly smooth on $C_0$. Shrink $C_0$ if necessary so that $C_0\subset A'_{R,w'}(\mathcal{O}_F)$, then $\varphi^{cc_0}=\varphi^{c}$ for all $c_0\in C_0$ and $c\in A'_{R,w'}$. We obtain that
\[B_\varphi^M(m,f_{1,w'})=\omega_\pi(z)B^M_\varphi(u_1\dot{w}_Mbc'u_2, f_{1,w'})\] is uniformly smooth as a function of $c'\in A'_{R,w'}$.
  
\end{proof}

\subsection{The final local coefficient formula}
Let us end this section by connecting the partial Bessel functions defined in section \ref{partial-Bessel-function} to partial Bessel integrals.
In section \ref{PrepLCF}, $Z_M^0U_M\backslash N'$ is identified with the Zariski open subspace of the 5-dimensional affine space $R'$. Let \[\dot{m}=\mathcal{N}(\dot{n})=\mathcal{N}(\underline{c}),\ \  \textrm{with}\  \dot{n}=\prod_{i=0}^4x_{\gamma_i}(c_i), \ \textrm{where}\  \underline{c}=(c_0,c_1,\cdots,c_4)\in R'.\] Let $\sigma$ be a supercuspidal representation of $M$ with central character $\omega_\sigma$. Choose $f\in C^\infty_c(M;\omega_\sigma)$ so that the Whittaker function $W_f$ of $\sigma$ is normalized as before.  Recall that the partial Bessel function is defined as
\[j_{\sigma,\kappa}(\underline{c})=j_{\sigma,\kappa}(\dot{n})=\int_{U_{M,\dot{n}}}W_f(\dot{m}u)\textbf{1}_{\overline{N}_{0,\kappa}}(zu^{-1}\overline{\dot{n}}uz^{-1})\psi^{-1}(u)du\]where $z=\alpha^\vee(\varpi^{d+g}\dot{x}_\alpha)$.

Now we define a function $\varphi=\varphi_\kappa$ on the image of $\mathcal{N}: N'\rightarrow M$ using $\textbf{1}_{\overline{N}_{0,\kappa}}$, and connect $j_{\sigma,\kappa}(\dot{n})$ with $B^M_\varphi(\dot{m},f)$.

\begin{prop}\label{Bessel-function-integral} There exists a function $\varphi=\varphi_\kappa$ on the image of $N$, depending only on $\kappa$, such that Assumption \ref{conjugate-invariance} is satisfied and
\[j_{\sigma,\kappa}(\dot{n})=B^M_{\varphi}(\dot{m},f).\] To emphasize that they are functions on $R'$, we also write this as \[j_{\sigma,\kappa}(\underline{c})=B^M_\varphi(\mathcal{N}(\underline{c}),f)\]    
\end{prop}

\begin{proof} First, by Corollary \ref{finiteetale2} we have a commutative diagram
\[\begin{tikzcd}
U_M\times R \arrow{r} \arrow[swap]{d}{\mathrm{Id}_{U_M}\times \mathcal{N}} & N' \arrow{d}{\mathcal{N}} \\%
U_M\times \mathcal{N}(R) \arrow{r}{\tau}& \mathcal{N}(N')
\end{tikzcd}\] We construct the function $\varphi$ by first constructing it on $\mathcal{N}(R)$.
It is not hard to see that the map $\phi:\overline{n} \mapsto n$  for $n\in R$ through the Bruhat decomposition $\dot{w}_0^{-1}n=mn'\overline{n}$ is a bijection up to a Zariski open dense subset of $\mathcal{N}(R)$.
Define $N_{x,\kappa}=\exp(\mathfrak{n}_{x,\kappa})$ where $\mathfrak{n}_{x,\kappa}$ is defined the same way as $\overline{\mathfrak{n}}_{x,\kappa}$, but replace $\Phi_N^-$ by $\Phi_N$, the set of roots in $N$. Then $\overline{n}\in \overline{N}_{0,\kappa}\Leftrightarrow n\in N_{0,\kappa}$. Next, by Corollary \ref{finiteetale2} and Lemma \ref{etale-fundamental-group}, the map $\mathcal{N}: N' \rightarrow M$ is \'etale of degree 2 onto its image, with a generator of the \'etale fundamental group given by the adjoint action of $\alpha^\vee(-1)$. On the other hand, $\overline{N}_{0,\kappa}$ is invariant under this action, so is $N_{0,\kappa}$. It follows that $\overline{\varphi}:= (\mathcal{N}\circ\phi)_*(\textbf{1}_{\overline{N}_{0,\kappa}})$ is a well-defined function on $\mathcal{N}(R)$.
Moreover, since $U_M\times \mathcal{N}(R)\overset{\tau}{\rightarrow}\mathcal{N}(N')$ is an isomorphism, and recall that we have an open compact subgroup $U_0$ of $U_M$ such that $\overline{N}_{0,\kappa}$ is invariant under the conjugation by $U_0$, the function 
\[\varphi=\varphi_\kappa:=\tau_*(\textbf{1}_{U_0}\times \overline{\varphi})=\tau_*(\textbf{1}_{U_0}\times (\mathcal{N}\circ\phi)_*(\textbf{1}_{\overline{N}_{0,\kappa-d-g}}))\] is well defined on $\mathcal{N}(N')\subset M$, where $\textbf{1}_{U_0}$ is the characteristic function of $U_0$. By construction, $\varphi$ is invariant under translation by elements in $Z_M$ of absolute value 1. As discussed in the beginning of Section \ref{partial-Bessel-integral}, $\varphi$ is independent of the choice of the representative $\dot{m}'$ of $\dot{m}$ under the $\Theta_M$-twisted conjugate action by $Z_M$. We may therefore set $\dot{m}'=\Theta_M(\alpha^\vee(\dot{x}_\alpha))\dot{m}\alpha^\vee(\dot{x}_\alpha)^{-1}$. 
By the $Z_M^0$- and $U_M$-equivariance of $\mathcal{N}$, the construction of $\varphi$ also implies that it is invariant under the $\Theta_M$-twisted conjugate action of $U_0$. Since $\alpha^\vee(t)\overline{N}_{0,\kappa}\alpha^\vee(t)^{-1}=\overline{N}_{0,\kappa+v_F(t)}$ by Lemma \ref{lem:root-subgroup-filtration-opposite-N}, we also have \[\varphi(\Theta_M(u)^{-1}\dot{m}'u)=\textbf{1}_{\overline{N}_{0,\kappa}}(zu^{-1}\overline{\dot{n}}u z^{-1}),\] where  $z=\alpha^\vee(\varpi^{d+g}\dot{x}_\alpha)$.
Finally, since  $U_{M,\dot{n}}=U_{M,\dot{m}}^{\Theta_M}=1$ up to a set of measure zero in our case, Assumption \ref{conjugate-invariance} also hold by Lemma \ref{trivial-stabilizer}, and we obtain that $j_{\sigma,\kappa}(\dot{n})=B^M_\varphi(\dot{m},f)$ as desired.
    
\end{proof}

 Recall that we have a map $\xi: M\rightarrow \mathrm{GL}_1\times\mathrm{GL}_6$ as in section \ref{GMpair} and we set $\sigma=\xi^*(1\otimes\pi)$. The following lemma shows how the partial Bessel functions vary under twisting by characters. 

\begin{lem}\label{twisted-partial-Bessel-function}Let $\dot{n}=\prod_{i=0}^5x_{\gamma_i}(c_i)$ be an orbit space representative of $Z_M^0U_M\backslash N'$, with $\underline{c}=(c_0,c_1,c_2,c_3,c_4)\in R'$, and $\dot{m}=\mathcal{N}(\dot{n})$. Suppose $\chi$ is a continuous character of $F^\times$, regarded as a character of $\mathrm{GL}_6(F)$ through the determinant. Set $\sigma_\chi=\xi^*(1\otimes(\pi\otimes \chi))$. Then
\[j_{\sigma_\chi,\kappa
}(\underline{c})=\chi^3(c_0^2-c_0c_1-c_2c_3c_4)j_{\sigma,\kappa}(\underline{c}).\]     
\end{lem}

\begin{proof} Recall that we have a commutative diagram
\[ \begin{tikzcd}
A_M\times M_D \arrow{r}{\overline{\xi}} \arrow[swap]{d}{} & \mathrm{GL}_1\times \mathrm{GL}_6 \arrow{d}{\mathrm{Id}} \\%
M \arrow{r}{\xi}& \mathrm{GL}_1\times\mathrm{GL}_6
\end{tikzcd}
\] such that $\overline{\xi}(\gamma_0^\vee(t), x)=(t^2,tx)$, and $\xi(\alpha_2^\vee(t))=(t,\mathrm{diag}(1,1,1,t,t,t))$. Since $\mathcal{N}(\dot{n})=\dot{m}=u_1w_Ma u_2$ where $a$ is of the form $a=\alpha_1^\vee(a_1)\alpha_2^\vee(a_2)\alpha_3^\vee(a_3)\alpha_4^\vee(a_4)\alpha_5^\vee(a_3)\alpha_6^\vee(a_1)$. Each $a_i$ is a rational function of $\underline{c}$. In particular, $a_2=c_0^2-c_0c_1-c_2c_3c_4$  by setting $c_5=1$ in Proposition \ref{BruhatDecomp}. Note that the image of $\alpha_i^\vee(a_i)$ lies in $M_D\simeq \mathrm{SL}_6$ for $i\neq 2$. As $\chi$ is defined through the determinant map, it follows that \[\chi(\xi(\dot{m}))=\chi^3(a_2)=\chi^3(c_0^2-c_0c_1-c_2c_3c_4)\] The assertion of the lemma follows immediately from this.
    
\end{proof}

Note that
$\omega_{\pi\otimes\chi}=\omega_\pi\chi^6$. By Proposition \ref{local-coeff-formula-1}, Lemma \ref{twisted-partial-Bessel-function}, and Proposition \ref{Bessel-function-integral},  we obtain 
\begin{prop}\label{local-coeff-formula-final}
     Let $\pi$ be a supercuspidal representation of $\mathrm{GL}_6(F)$, and $\chi$ a continuous character of $F^\times$, identified as a character of $\mathrm{GL}_6(F)$ through the determinant. Let $\sigma=\xi^*(1\otimes\pi)$ and $\sigma_\chi=\xi^*(1\otimes(\pi\otimes\chi))$ as before.  Take $f\in C^\infty_c(M;\omega_\sigma)$ such that $W_f$ is a non-zero Whittaker function of $\sigma$, normalized so that $W_f(e)=1$.  Then
$$
C_\psi(s,\sigma_\chi)^{-1}
=\gamma(4 s, \omega^2_\pi\chi^{12}, 
\psi^{-1})\int_{(F^\times)^5}j_{\sigma_\chi,\kappa}(\underline{c})\omega_\pi^{-2}\chi^{-12}(\frac{c_1-c_0 }{c_0^2 - c_0c_1 - c_2c_3c_4})\vert c_1-c_0 \vert^{-4s}$$$$\cdot\vert c_0^2-c_0c_1-c_2c_3c_4\vert^{5s+\frac{11}{2}}\vert c_3\vert^2\vert c_4\vert^4\prod_{i=0}^4dc_i$$$$=
\gamma(4 s, \omega^2_\pi\chi^{12}, 
\psi^{-1})\int_{(F^\times)^5}B^M_\varphi(\mathcal{N}(\underline{c}),f)\omega_\pi^{-2}\chi^{-12}(c_1-c_0)\vert c_1-c_0 \vert^{-4s}$$$$\cdot\omega_\pi^2\chi^{15}(c_0^2-c_0c_1-c_2c_3c_4)\vert c_0^2-c_0c_1-c_2c_3c_4\vert^{5s+\frac{11}{2}}\vert c_3\vert^2\vert c_4\vert^4\prod_{i=0}^4dc_i
$$
for sufficiently large $\kappa$, where $\gamma(4 s, \omega^2_\pi\chi^{12}, 
\psi^{-1})$ is the Tate $\gamma$-factor associated to the character $\omega_{\pi\otimes\chi}^2=\omega_\pi^2\chi^{12}$.
\end{prop}

\subsection{Proof of main theorem} By the structure theorem of representations of p-adic groups, it suffices to prove the stability of local coefficients for $\psi$-generic supercuspidal representations of $\mathrm{GL}_6(F)$. Suppose $\pi_1,\pi_2$ are $\psi$-generic supercuspidal representations of $\mathrm{GL}_6(F)$ with the same central character $\omega_\pi:=\omega_{\pi_1}=\omega_{\pi_2}$, and let $\sigma_i=\xi^*(1\otimes\pi_i)$, $(i=1,2)$. Set $\omega_{\sigma}:=\omega_{\sigma_1}=\omega_{\sigma_2}$. Given a continuous character $\chi$ of $F^\times$, set $\sigma_{i,\chi}=\xi^*(1\otimes(\pi_i\otimes\chi))$. We will show that
\[C_\psi(s,\sigma_{1,\chi})=C_\psi(s,\sigma_{2,\chi})\] when $\chi$ is sufficiently ramified.

Choose $f_1,f_2\in C^\infty_c(M;\omega_\sigma)$ to be matrix coefficients of $\sigma_1$ and $\sigma_2$ respectively, normalized so that their corresponding Whittaker functions $W_{f_i}$ satisfy that $W_{f_1}(e)=W_{f_2}(e)=1$. Let $\kappa$ be sufficiently large such that Proposition 
\ref{Asymptoticformula} holds for $f_1,f_2$ and the auxiliary function $f_0$, and also that
Proposition \ref{local-coeff-formula-final} holds. Apply the asymptotic expansion formula in Proposition \ref{Asymptoticformula} to the integral formula in Proposition \ref{local-coeff-formula-final}, we obtain
\[C_\psi(s,\sigma_{1,\chi})^{-1}-C_\psi(s,\sigma_{2,\chi})^{-1}=\gamma(4s, \omega_\pi^2\chi^{12},\psi^{-1})D_\chi(s)\]
where $$D_\chi(s)=\int_{(F^\times)^5}(B^M_\varphi(\mathcal{N}(\underline{c}),f_1)-B^M_\varphi(\mathcal{N}(\underline{c}),f_2))\omega_\pi^{-2}\chi^{-12}(c_1-c_0)\vert c_1-c_0 \vert^{-4s}$$$$\cdot\omega_\pi^2\chi^{15}(c_0^2-c_0c_1-c_2c_3c_4)\vert c_0^2-c_0c_1-c_2c_3c_4\vert^{5s+\frac{11}{2}}\vert c_3\vert^2\vert c_4\vert^4\prod_{i=0}^4dc_i$$
It suffices to show that $D_\chi(s)=0$ when $\chi$ is sufficiently ramified. By Corollary \ref{finiteetale3}, the map $\mathcal{N}: R'\rightarrow M$ is an isomorphism onto its image in $M$, up to a Zariski dense subset. This allows us to translate the integral over $R'$ to the image $\mathcal{N}(R')$. To be precise, the map $\mathcal{N}$ can be identified as $\underline{c}=(c_0,c_1,c_2,c_3,c_4)\mapsto \underline{t}=(t_1',t_2',t_3',t_5',t_6') $ with the determinant of the Jacobian $-c_3=-\frac{t_3'}{t_1'}$. By the arguments in the proof of Proposition \ref{uniform-smooth}, if we
write $\mathcal{N}(\underline{c})=\dot{m}=u_1\dot{w}_M a u_2$, then \[a=\alpha_1^\vee(\frac{t_2'}{t_1'})\alpha_2^\vee(t_2')\alpha_3^\vee(\frac{t_2'^2}{t_3'})\alpha_4^\vee(\frac{t_2'^3}{t_2'-t_1'^2t_5't_6'})\alpha_5^\vee(\frac{t_2'^2}{t_3'})\alpha_6^\vee(\frac{t_2'}{t_1'})\] Let $a_1'=\frac{t_2'}{t_1'}, a_2'=t_2' , a_3'=\frac{t_2'^2}{t_3'}, a_4'=\frac{t_2'^3}{t_2'-t_1'^2t_5't_6'}, a_5'=t_5', a_6'=t_6'$. Then $a_6'=\frac{a_1'^2(a_4'-a_2'^2)}{a_2'a_4'a_5'}$. Set $\underline{a}=(a_1',a_2',a_3',a_4',a_5')$. The map $\underline{t}\mapsto \underline{a}$ is birational with the determinant of the Jacobian given by $-\frac{a_3'^2a_4'^2a_5'}{a_2'^4}$. By further assuming that $a_4'-a_2'^2\neq 0$, it becomes an isomorphism. Thus, the composition $\underline{c}\mapsto \underline{t}\mapsto\underline{a}$ is a bijection up to a Zariski dense subset, and we have
\[c_1-c_0=t_1't_5'=\frac{a_2'a_5'}{a_1'}, c_3=\frac{t_3'}{t_1'}=\frac{a_1'a_2'}{a_3'}, c_4=-t_1'=-\frac{a_2'}{a_1'}\] By Corollary \ref{finiteetale3}, 
the determinant of the Jacobian of this map is $-\frac{t_3'}{t_1'}\cdot(-\frac{a_3'^2a_4'^2a_5'}{a_2'^4})=\frac{a_1'a_3'a_4'^2a_5'}{a_2'^3}$.
Consequently,
$$D_\chi(s)=\int_{(F^\times)^5}(B^M_\varphi(\dot{m},f_1)-B^M_\varphi(\dot{m},f_2))\omega_\pi^{-2}\chi^{-12}(\frac{a_2'a_5'}{a_1'})\vert \frac{a_2'a_5'}{a_1'} \vert^{-4s}$$$$\cdot\omega_\pi^2\chi^{15}(a_2')\vert a_2'\vert^{5s+\frac{11}{2}}\vert \frac{a_1'a_2'}{a_3'}\vert^2\vert \frac{a_2'}{a_1'}\vert^4\vert\frac{a_2'^3}{a_1'a_3'a_4'^2a_5'}\vert \prod_{i=1}^5da'_i$$
$$=\int_{F^\times}\int_{\mathrm{Pr}_A(\mathrm{Im}\mathcal{N})}(B^M_\varphi(\dot{m},f_1)-B^M_\varphi(\dot{m},f_2))\omega_\pi^{-2}\chi^{-12}(\frac{a_2'}{a_1'})\vert \frac{a_2'}{a_1'} \vert^{-4s}$$$$\cdot\omega_\pi^2\chi^{15}(a_2')\vert a_2'\vert^{5s+\frac{11}{2}}\vert\frac{a_2'^9}{a_1'^3a_3'^3a_4'^2}\vert \prod_{i=1}^4da'_i\cdot\omega_{\pi}^{-2}\chi^{-12}(a_5')\vert a_5'\vert^{-4s-1}d a_5'$$
By Proposition \ref{Asymptoticformula}, the first term in the asymptotic expansion of $B^M_\varphi(\dot{m},f_i),(i=1,2)$ depends only on $\omega_\pi$ and the auxiliary function $f_0$, hence they cancel by our assumption that $\omega_{\pi_1}=\omega_{\pi_2}$. For the remaining terms, Proposition \ref{Asymptoticformula} implies that for each $w'\in B(M)$ with $d_B(w',e)\ge 1$, one can find $f_{i,w'}\in C^\infty_c(\Omega_{w'};\omega_\sigma)$, such that
\[B^M_\varphi(\dot{m},f_1)-B^M_\varphi(\dot{m},f_2)=\sum_{w'\in B(M), d_B(w',e)\ge 1}(B^M_\varphi(\dot{m},f_{1,w'})-B^M_\varphi(\dot{m},f_{2,w'}))\]
As $A_{w_M}^{w'}A_{w'}\subset A$ is open of finite index, we can replace the inner integral over
$\mathrm{Pr}_A(\mathrm{Im}\mathcal{N})$ as a double integral over
$A_{w_M}^{w'}\cap\mathrm{Pr}_A(\mathrm{Im}\mathcal{N})$ and $A_{R',w'}$. Consequently, 
the inner integral is equal to
$$\sum_{w'\in B(M),d_B(w',e)\ge 1}\int_{A_{w_M}^{w'}\cap\mathrm{Pr}_A(\mathrm{Im}\mathcal{N})}\int_{A_{R',w'}}(B^M_\varphi(\dot{m},f_{1,w'})-B^M_\varphi(\dot{m},f_{2,w'}))$$$$\cdot\omega_\pi^{-2}\chi^{-12}(\frac{a_2'}{a_1'})\vert \frac{a_2'}{a_1'} \vert^{-4s}\omega_\pi^2\chi^{15}(a_2')\vert a_2'\vert^{5s+\frac{11}{2}}\vert\frac{a_2'^9}{a_1'^3a_3'^3a_4'^2}\vert \prod_{i=1}^4da'_i$$
$$=\sum_{w'\in B(M),d_B(w',e)\ge 1}\int_{A_{w_M}^{w'}\cap \mathrm{Pr}_A(\mathrm{Im}\mathcal{N})}\int_{A_{R',w'}}(B^M_\varphi(\dot{m},f_{1,w'})-B^M_\varphi(\dot{m},f_{2,w'}))$$$$\cdot\omega_\pi^{-2}(\frac{a_2'}{a_1'})\vert \frac{a_2'}{a_1'} \vert^{-4s}\omega_\pi^2(a_2')\vert a_2'\vert^{5s+\frac{11}{2}}\vert\frac{a_2'^9}{a_1'^3a_3'^3a_4'^2}\vert \chi({a_1'}^{12}{a_2'}^3)\prod_{i=1}^4da'_i$$
By Proposition \ref{uniform-smooth}, 

 \[(B^M_\varphi(u_1\dot{w}_Mbc'u_2,f_{1,w'})-B^M_\varphi(u_1\dot{w}_mbc'u_2,f_{2,w'}))\omega_{\pi}^2(a_1')\vert \frac{a_2'^{s+\frac{29}{2}}}{a_1'^{-4s+3}a_3'^3a_4'^2}\vert\]
 is uniformly smooth as a function of $c'\in A_{R',w'}$. As discussed in the proof of Proposition \ref{uniform-smooth}, $A_{R',w'}\neq \emptyset$ for all $w'\in B(M)$ with $d_B(w',e)\ge 1$. Moreover, each
 $w'\in B(M)$ with $d_B(w',e)\ge 1$ can be written as $w'=w_Mw_L^{-1}$ for some proper Levi subgroup $L\subsetneq M$, and $A_{w'}=(\cap_{\beta\in \Delta_L}\ker \beta)^\circ$, where $\Delta_L\subsetneq\Delta_M=\{\alpha_1,\alpha_3, \alpha_4,\alpha_5,\alpha_6\}$.
 If we write $c'=\alpha_1^\vee(a'_1)\alpha_2^\vee(a'_2)\alpha_3^\vee(a'_3)\alpha_4^\vee(a'_4)\alpha_5^\vee(a'_2)\alpha_6^\vee(a'_1)$,
 then \[\alpha_1(c')=\alpha_6(c')={a_1'}^2{a_3'}^{-1}, \alpha_3(c')=\alpha_5(c')={a_3'}^2{a_1'}^{-1}{a_4'}^{-1}, \alpha_4(c')={a_4'}^2{a_3'}^{-2}{a_2'}^{-1}\]
 From this, it follows that
 $a'_2$-component of $c'$ is not trivial for all such possible choices of $\Delta_L$. When $\chi$ is sufficiently ramified, one can take a small open compact subgroup $C_0$ of $\mathcal{O}_F^\times$ such that $\chi(a'_2x)\neq \chi(a_2')$ for all $x\in C_0$. As a result, inner integral over $A_{R',w'}$ vanishes, and consequently, $D_\chi(s)=0$. This concludes the proof of the main theorem.

\bibliographystyle{plain}
\bibliography{biblio}

\end{document}